\documentclass[leqno]{amsart}

\usepackage[english]{babel} 									
\usepackage[T1]{fontenc}							
\usepackage[utf8]{inputenx}
				
\usepackage{mathtools}								
\usepackage{empheq}									
\usepackage{enumitem}								

\usepackage{amssymb}								
\usepackage{mathrsfs}								
\usepackage{bm}                                     
				
\usepackage{verbatim}                               
\usepackage{iflang}									
\usepackage{xifthen}								

\usepackage[final]{pdfpages}						
\usepackage{todonotes}								
\usepackage{aliascnt}								
\usepackage{needspace}								
\usepackage{refcount}

\usepackage{subcaption}			
\usepackage{placeins} 			
\usepackage{grffile}			
\usepackage[all]{xy}			

\usepackage{transparent}
\usepackage{color}				

\usepackage{tikz}               
\usetikzlibrary{math,arrows.meta, cd}
\DeclareGraphicsExtensions{.pdf,.png,.jpg}

\usepackage{graphicx}			
 
\usepackage{diagbox} 

\usepackage[babel]{csquotes}							

\usepackage{nameref}									
\usepackage[colorlinks	=	true,						
          	 	linkcolor	=	red,
            	urlcolor	=	red,
            	citecolor	=	red]%
            	{hyperref}
\usepackage{bookmark}									
\usepackage{cleveref}									
\numberwithin{equation}{section}

\newtheorem{theorem}{Theorem}[section]

\newtheorem{corollary}[theorem]{Corollary}
\newtheorem{lemma}[theorem]{Lemma}
\newtheorem{proposition}[theorem]{Proposition}

\newtheorem*{theorem*}{Theorem}

{\bf}{\it}

\newtheorem{definition}[theorem]{Definition}

\theoremstyle{remark}
\newtheorem{remark}[theorem]{Remark}

\newtheorem*{remark*}{Remark}

\theoremstyle{remark}

\newcommand{\N}{\mathbb{N}}
\renewcommand{\S}{\mathbb{S}}

\newcommand{\R}{\mathbb{R}}

\DeclareMathOperator{\diam}{diam}

\DeclareMathOperator{\loc}{loc}

\DeclareMathOperator{\LIP}{LIP}

\allowdisplaybreaks

\newcommand{\norm}[1]{ \left\| #1 \right\| }	
\newcommand{\abs}[1]{ \left| #1 \right| }           
\newcommand{\apmd}[2][]{							
	\ifthenelse{\equal{#1}{}}%
					{ \operatorname{N}_{#2}	}%
					{ \operatorname{N}_{#1,#2} 	}}

\newcommand{\aint}[2][]{
	\ifthenelse{\equal{#1}{}}%
					{%
\mathchoice%
      {\mathop{\kern 0.2em\vrule width 0.6em height 0.69678ex depth -0.58065ex
              \kern -0.8em \intop}\nolimits_{\kern -0.45em#2}^{#1}}%
      {\mathop{\kern 0.1em\vrule width 0.5em height 0.69678ex depth -0.60387ex
              \kern -0.6em \intop}\nolimits_{#2}^{#1}}%
      {\mathop{\kern 0.1em\vrule width 0.5em height 0.69678ex depth -0.60387ex
              \kern -0.6em \intop}\nolimits_{#2}^{#1}}%
      {\mathop{\kern 0.1em\vrule width 0.5em height 0.69678ex depth -0.60387ex
              \kern -0.6em \intop}\nolimits_{#2}^{#1}}}%
					{%
\mathchoice%
      {\mathop{\kern 0.2em\vrule width 0.6em height 0.69678ex depth -0.58065ex                                              
              \kern -0.8em \intop}\nolimits_{\kern -0.45em#1}^{#2}}%
      {\mathop{\kern 0.1em\vrule width 0.5em height 0.69678ex depth -0.60387ex
              \kern -0.6em \intop}\nolimits_{#1}^{#2}}%
      {\mathop{\kern 0.1em\vrule width 0.5em height 0.69678ex depth -0.60387ex
              \kern -0.6em \intop}\nolimits_{#1}^{#2}}%
      {\mathop{\kern 0.1em\vrule width 0.5em height 0.69678ex depth -0.60387ex
              \kern -0.6em \intop}\nolimits_{#1}^{#2}}}}

\newcommand{\set}[1]{\{\, {#1} \,\}}

\newcommand{\Set}[1]{\left\{\, {#1} \,\right\}}

\newcommand{\vol}{\mathrm{vol}}

\usepackage{thmtools}
\usepackage{thm-restate}

\newcommand{\C}{\mathbb{C}}
\newcommand{\CP}{\mathbb{C}P}

\newcommand{\SL}{\mathrm{SL}}

\newcommand{\FS}{\mathrm{FS}}

\title[Rescaling principle and hyperbolicity]{A rescaling principle for quasiregular curves with applications to hyperbolicity}
\author{Toni Ikonen}
\author{Jonathan Pim}

\address{Department of Mathematics\\ University of Fribourg\\  Chemin du Mus\'ee 23\\  1700 Fribourg, Switzerland}
\email{toni.ikonen@unifr.ch}

\address{Department of Mathematics and Statistics, P. O. Box 68 (Pietari Kalmin katu 5), FI-00014 University of Helsinki, Finland}
\email{jonathan.pim@helsinki.fi}

\thanks{The first author was supported by the Swiss National Science Foundation grant 212867. The second author was supported by the Research Council of Finland, project number 332671 and Center of Excellence FiRST}

\begin{document}

\begin{abstract}
    We prove a Miniowitz--Zalcman rescaling principle for quasiregular curves into calibrated manifolds. We have two main applications.
    
    First, we introduce Brody hyperbolicity adapted to our setting and prove its equivalence to the normality of the family of quasiregular curves from the Euclidean unit ball into the target. When normality holds, we quantify the local modulus of continuity for quasiregular curves using an injectivity radius lower bound and a sectional curvature upper bound of the target.
    
    Second, in the special case of conformal curves into closed calibrated manifolds, we prove the equivalence of Kobayashi and Brody hyperbolicity. This answers a question posed by Broder--Iliashenko--Madnick. As an intermediate result, we prove an analogue of Marty's theorem from complex analysis in this setting.
    
    Additionally, we construct new examples of non-constant entire quasiregular curves factoring through a Special Lagrangian submanifold of a closed Calabi--Yau manifold, an associative submanifold of a closed $G_2$ manifold, and four-dimensional analogues thereof, providing obstructions to Brody hyperbolicity.
\end{abstract}

\maketitle

{
    \hypersetup{linkcolor=black}
    \setcounter{tocdepth}{1}
    \tableofcontents
}

\section{Introduction}

\subsection{Background and the main theorem}

    In this manuscript, we investigate rescaling principles and their applications in the calibrated geometries of Harvey and Lawson \cite{harvey-lawson-1982-calibrated-geometries}.
    A classical instance of such a principle is \emph{Brody's theorem} \cite{brody-1978-compact-manifolds-and-hyperbolicity} in complex geometry:
    \emph{for a closed complex manifold $N$, the family of holomorphic maps from the unit disk $\mathbb{D} \subset \mathbb{C}$ into $N$ is normal if and only if there exists no non-constant entire holomorphic map $\mathbb{C} \to N$.}
    This connection between normality and non-existence of entire holomorphic maps is sharpened by \emph{Zalcman's lemma} \cite{zalcman-1975-a-heuristic-principle-in-complex-function-theory}; see also Royden's work \cite{royden-1971-remarks-on-the-kobayashi-metric} on Kobayashi hyperbolicity \cite{kobayashi-1970-hyperbolic-manifolds-and-holomorphic-mappings}.

    The rescaling principle has also been influential in the theories of various generalizations of holomorphic mappings. These include \emph{quasiregular mappings} and the Miniowitz--Zalcman lemma \cite{miniowitz-1982-normal-families-of-quasimeromorphic-mappings}, and Kobayashi hyperbolicity \cite{kruglikov-overholt-1999-pseudoholomorphic-mappings-and-kobayashi-hyperbolicity} for Gromov's \emph{pseudoholomorphic curves} \cite{gromov-1985-pseudo-holomorphic-curves-in-symplectic-manifolds}. More recently, Broder, Iliashenko, and Madnick developed in \cite{broder-iliashenko-madnick-2025-hyperbolicity-and-schwarz-lemmas-in-calibrated-geometry} a theory of hyperbolicity in the context of the \emph{calibrated geometries} of Harvey and Lawson \cite{harvey-lawson-1982-calibrated-geometries}. In particular, they formulated notions of Kobayashi and Brody hyperbolicity using \emph{Smith maps} \cite{smith-2012-a-theory-of-multiholomorphic-maps} and asked if they are equivalent for closed manifolds. In this manuscript, we give a positive answer to their question. For the proof and other applications, we establish a generalized Miniowitz--Zalcman rescaling principle for \emph{quasiregular curves} \cite{pankka-quasiregular-curves}, a class which includes quasiregular mappings, pseudoholomorphic curves, Smith maps, and mappings satisfying Monge--Ampère type inequalities \cite{iwaniec-verchota-vogel-2002-the-failure-of-rank-one-connections} or generalized Beltrami-type equations \cite{rosay-2010-uniqueness-in-rough-almost-complex-structures-and-differential-inequalities}.
    
    To define quasiregular curves, we first need a definition for a calibrated manifold: given integers $2 \leq n \leq m$, we say that a pair $(N,\omega)$ is an \emph{$n$-calibrated $m$-manifold} if $N$ is an $m$-dimensional (Riemannian) manifold and $\omega$ is a strongly non-vanishing \(n\)-calibration, i.e. a closed differential $n$-form on $N$ whose comass is bounded from above by one and below by a positive constant.
    
    Now, given an oriented (Riemannian) $n$-manifold \(M\) and $K \geq 1$, a continuous map $F \colon M \to ( N, \omega )$ is a \emph{\(K\)-quasiregular curve} if $F$ is in the local Sobolev space $W^{1,n}_{\text{loc}}(M,N)$ and
    \begin{align}\label{equation-quasiregular-curve-distortion-inequality}
        ( \|\omega\| \circ F ) \|DF\|^n \leq K \star F^{*}\omega
        \quad\text{almost everywhere in}\ M.
    \end{align}
    Here \(\norm{\omega}\) is the \emph{pointwise comass} of \(\omega\), and \(\star \colon \Omega^n(M) \to \Omega^0(M)\) is the Hodge star operator on \(M\).

    In this terminology, \(K\)-quasiregular mappings \(M \to N\) are \(K\)-quasiregular curves into \((N,\vol_N)\) where \(\vol_N\) is the Riemannian volume form; see \cite{rickman-1993-quasiregular-mappings} for the definition. Similarly, pseudoholomorphic curves into a symplectic manifold \((N,\omega_{\mathrm{symp}})\) equipped with an \(\omega_{\mathrm{symp}}\)-tame almost complex structure are \(1\)-quasiregular curves into \((N,\omega_{\mathrm{symp}})\); see \cite[Chapter II]{mcduff-salamon-2012-j-holomorphic-curves-and-symplectic-topology}. Conversely, every such \(1\)-quasiregular curve is pseudoholomorphic by elliptic regularity theory \cite[Appendix B.4]{mcduff-salamon-2012-j-holomorphic-curves-and-symplectic-topology}.
    See \cite{pankka-quasiregular-curves, heikkila-pankka-prywes-2023-quasiregular-curves-of-small-distortion-in-product-manifolds, ikonen-pankka-2024-liouvilles-theorem-in-calibrated-geometries,heikkila-2024-quasiregular-curves-and-cohomology} for further examples of quasiregular curves.

    We now formulate our main theorem.
\begin{theorem}[Miniowitz--Zalcman rescaling principle]\label{theorem-rescaling-principle}
    Let $(N,\omega)$ be a closed $n$-calibrated $m$-manifold, and let $\mathcal{F}$ be a family of $K$-quasiregular curves $M \to (N,\omega)$ for an oriented $n$-manifold $M$ and \(K \geq 1\).
    If $\mathcal{F}$ is not a normal family at a point $x \in M$, then there exists a non-constant $K$-quasiregular curve $G \colon \R^n \to (N,\omega)$. More precisely, there exist sequences $v_j \in T_{x} M$ with $v_j \rightarrow 0$, constants $\rho_j \rightarrow 0^{+}$, and $F_j \in \mathcal{F}$ such that the sequence
        \begin{align*}
            z \mapsto G_j( z ) = F_j( \mathrm{exp}_{x}( v_j + \rho_j z ) )
        \end{align*}
        converges locally uniformly to a non-constant $K$-quasiregular curve $G \colon T_x M \to (N,\omega)$. Moreover, the pairs $(v_j, \rho_j)_j$ can be chosen so that the limit map $G$ is uniformly Hölder continuous.
\end{theorem}
    In the statement above, $\mathrm{exp}_x \colon U \to M$ is the exponential map defined on an open neighbourhood $U \subset T_xM$ of the origin. Notice that since $v_j \rightarrow 0$ and $\rho_j \rightarrow 0^{+}$, the domain of definition of $G_j$ converges to the tangent space $T_x M$ as $j \to \infty$. The local uniform convergence of mappings with variable domains is understood in the standard way: for every compact set $C \subset T_x M$ and large enough $j$, the restrictions $G_j|_C$ are well-defined and converge uniformly to the limiting map. We follow this convention here and below in the manuscript. 

    Our proof of \Cref{theorem-rescaling-principle} is based on rescaling at points where the family concentrates a definite amount of energy. This differs from the standard proof of the Miniowitz--Zalcman rescaling principle (\cite[Lemma~1]{miniowitz-1982-normal-families-of-quasimeromorphic-mappings} or \cite[Theorem~19.7.3]{Iwaniec-Martin-Geometric-Function-Theory-Non-Linear-Analysis}) based on local Hölder estimates. The critical lower bound for the energy depends on the geometry of the target, through its isoperimetric profile. Observe that the relevant profiles are in codimension $m-n$ and $m-n+1$ which may be large and are therefore naturally formulated using the theory of integral currents. These methods are flexible enough that we prove a version of \Cref{theorem-rescaling-principle} for non-compact targets which satisfy a sectional curvature upper bound and an injectivity radius lower bound; see \Cref{theorem-rescaling-principle-equiboundedness-build-in}.

\subsection{Normality, Brody hyperbolicity and quasiregular ellipticity}
    As one of our main applications of \Cref{theorem-rescaling-principle}, we prove the following:
\begin{theorem}[Brody's theorem for calibrated manifolds]\label{theorem-qr-ellipticity-equivalent-to-normality}
    Let $(N,\omega)$ be a closed $n$-calibrated $m$-manifold and $K \geq 1$. Then the following are equivalent.
    \begin{enumerate}
        \item every $K$-quasiregular curve $\R^n \to (N,\omega)$ is constant;

        \item the family of $K$-quasiregular curves $M \to (N,\omega)$ is a normal family, for each oriented $n$-manifold \(M\);
        
        \item the family of $K$-quasiregular curves $\R^n \supset \mathbb{B} \to (N,\omega)$ is a normal family for the unit ball $\mathbb{B}$.
    \end{enumerate}
\end{theorem}
    Property (1) is associated with two complementary notions: if \Cref{theorem-qr-ellipticity-equivalent-to-normality} (1) holds, we say that $(N,\omega)$ is \emph{Brody $K$-hyperbolic} in analogy to \cite{brody-1978-compact-manifolds-and-hyperbolicity}. Otherwise $(N,\omega)$ is \emph{$K$-quasiregularly elliptic} following  \cite{heikkila-2024-quasiregular-curves-and-cohomology}. Quasiregular ellipticity for mappings was introduced by Bonk--Heinonen in \cite{bonk-heinonen-2001-quasiregular-mappings-and-cohomology}. For mappings, the complete classification of closed quasiregularly elliptic \(3\)-manifolds is due to Jormakka \cite{jormakka-1988-the-existence-of-quasiregular-mappings-from-r3-to-closed-orientable-3-manifolds}, following the resolution of Thurston's Geometrization Conjecture by Perelman. 
    Recently, the classification for closed \(4\)-manifolds was completed; see Heikkilä--Pankka and Piergallini--Zuddas \cite{heikkila-pankka-2025-de-rham-algebras-of-closed-quasiregularly-elliptic-manifolds-are-euclidean,piergallini-zuddas-2021-branched-coverings-of-basic-4-manifolds} for the simply connected case and Manin--Prywes \cite{manin-prywes-2025-elliptic-quasiregularly-elliptic-manifolds} for the general case. 
    In \Cref{section-examples}, we show that many of these manifolds occur as calibrated submanifolds of closed manifolds with special holonomy. Moreover, by applying Bryant's construction of Calabi--Yau manifolds \cite{bryant-2000-calibrated-embeddings-in-the-special-lagrangian-and-coassociative-cases}, we show that every closed manifold in Jormakka's classification occurs as a Special Lagrangian in a non-compact Calabi--Yau manifold, cf. \Cref{corollary-reaping-the-harvest}. These provide new examples of quasiregular ellipticity for curves.

    Recently, the second author and Pankka showed in \cite{pankka-pim-2025-nodal-resolution-of-quasiregular-curves-via-bubble-trees} that \Cref{theorem-qr-ellipticity-equivalent-to-normality} (2) is closely related to non-trivial \emph{bubbling} of quasiregular curves; see also the bubbling phenomenon for harmonic maps by Sacks--Uhlenbeck \cite{sacks-uhlenbeck-1981-the-existence-of-minimal-immersions-of-2-spheres}, pseudoholomorphic curves (see \cite{gromov-1985-pseudo-holomorphic-curves-in-symplectic-manifolds} and \cite[Section 4]{mcduff-salamon-2012-j-holomorphic-curves-and-symplectic-topology}), Smith maps \cite{cheng-karigiannis-madnick-2020-bubble-three-convergence-of-conformally-cross-product-preserving-maps}, and quasiregular mappings \cite{pankka-souto-2023-bubbling-of-quasiregular-maps}. Indeed, if there exists a family $\mathcal{F}$ of $K$-quasiregular curves $M \to ( N, \omega )$ with \emph{bounded energy} and which is not normal, then there exists a non-constant $K$-quasiregular curve $\S^n \to ( N, \omega )$ by \cite{pankka-pim-2025-nodal-resolution-of-quasiregular-curves-via-bubble-trees}. Conversely, if there exists a non-constant $K$-quasiregular curve $\S^n \to ( N, \omega )$, such a family $\mathcal{F}$ is formed by the orbit of the curve under the action of the conformal automorphism group of $\S^n$.
    
    We expect that families of non-constant quasiregular curves $\S^n \to ( N,\omega)$ can be used to define a non-trivial homological and geometric invariant of the target $(N,\omega)$. Indeed, the Hurewicz homomorphism $\pi_n(N) \to H_n(N, \mathbb{Z})$ maps the homotopy classes of such curves $\S^n \to ( N, \omega )$ into a \emph{cone} of non-trivial homology classes of \emph{positive $\omega$-currents} of Harvey--Lawson \cite{harvey-lawson-1982-calibrated-geometries}, cf. \Cref{section-homological-non-triviality}. As a related invariant, Joyce proposed in \cite{joyce-2002-on-counting-special-lagrangian-homology-3-spheres} a suitable \emph{weighted count} of the homology classes of Special Lagrangian rational homology $3$-spheres on Calabi--Yau manifolds. To see the relation between these concepts, we recall that if such a rational homology sphere in $(N,\omega_{\SL})$, where $\omega_{\SL}$ is the Special Lagrangian calibration, has a finite fundamental group, its universal cover is diffeomorphic to $\S^3$ and thus the rational homology sphere is the image of a quasiregular curve $\S^3 \to ( N, \omega_{\SL} )$. This discussion readily extends to associative rational homology spheres in $G_2$ manifolds \cite{joyce-2018-conjectures-on-counting-associative-4-folds-in-g2-manifolds}. For the closely related theory of Gromov--Witten invariants for pseudoholomorphic curves, we refer to \cite[Chapter 7]{mcduff-salamon-2012-j-holomorphic-curves-and-symplectic-topology}.

    In \Cref{theorem-liouville-property-equivalent-normality} below, we prove a version of \Cref{theorem-qr-ellipticity-equivalent-to-normality} for non-compact targets which have a sectional curvature upper bound and an injectivity radius lower bound.

\subsection{Applications to notions of hyperbolicity for conformal curves}\label{section-kobayashi-hyperbolicity}
    In this section, we focus on a calibrated manifold $(N,\omega)$ and \emph{$1$-quasiregular curves} $f \colon M \to ( N, \omega )$ that map into $\{ x \in N \colon \|\omega\|(x) = 1 \}$; these are also called \emph{conformal curves} in \cite{ikonen-pankka-2024-liouvilles-theorem-in-calibrated-geometries}, \emph{calibrated curves} in \cite{heikkila-pankka-prywes-2023-quasiregular-curves-of-small-distortion-in-product-manifolds}, and \emph{Smith maps} in \cite{smith-2012-a-theory-of-multiholomorphic-maps}. We adopt the term \emph{conformal curve} for this manuscript.

    Every conformal curve is locally weakly $n$-harmonic and thus $\mathcal{C}^{1,\alpha}_{\loc}$-regular by the regularity theorem of Hardt--Lin \cite{hardt-lin-1987-mappings-minimizing-the-lp-norm-of-the-gradient}. Using this fact, Broder--Iliashenko--Madnick introduced in \cite{broder-iliashenko-madnick-2025-hyperbolicity-and-schwarz-lemmas-in-calibrated-geometry} an \emph{extended Finslerian metric} on each calibrated manifold. To set up the definition, we denote by $\mathcal{F}_{1}(\omega)$ the family of conformal curves $\mathbb{B} \to ( N, \omega )$. We define for each tangent vector $v \in TN$,
    \begin{align*}
        \mathcal{K}_{(N,\omega)}(v)
        \coloneqq
        \inf\left\{
            a > 0
            \mid
            \text{$f \in \mathcal{F}_1(\omega)$, $(Df)_0(w) = a^{-1}v$ for some $|w| = 1$}
        \right\}.
    \end{align*}
    The infimum over the empty set is understood to be infinite. Since constant maps are allowed, obviously $\mathcal{K}_{(N,\omega)}$ vanishes on the zero section. We also note that the functional is positively homogeneous \cite{broder-iliashenko-madnick-2025-hyperbolicity-and-schwarz-lemmas-in-calibrated-geometry}. Using the above, we define the following pseudodistance
    \begin{align*}
        d_{(N,\omega)}(x,y)
        \coloneqq
        \inf\left\{
            \int_{ \gamma }^{*}
                \mathcal{K}_{(N,\omega)}( \gamma'(t) )
            \,d t
        \right\},
    \end{align*}
    where the infimum is taken over piecewise $\mathcal{C}^1$-regular curves joining $x$ to $y$. Since the measurability of $t \mapsto \mathcal{K}_{(N,\omega)}( \gamma'(t) )$ is not obvious in this generality, we use the upper integral in the definition of $d_{(N,\omega)}$ (see e.g. \cite[2.4.2]{federer-1969-geometric-measure-theory} for a definition).  
    
    Following \cite{broder-iliashenko-madnick-2025-hyperbolicity-and-schwarz-lemmas-in-calibrated-geometry}, we define three notions of hyperbolicity closely related to the discussion so far:
    \begin{enumerate}
        \item $(N,\omega)$ is \emph{$\omega$-hyperbolic} if every conformal curve $\R^n \to ( N, \omega )$ is constant, i.e. Brody $1$-hyperbolicity holds for this subclass of quasiregular curves;
        \item $(N,\omega)$ is \emph{$K_\omega$-hyperbolic} if $d_{ (N,\omega) }(x,y) > 0$ for every $x \neq y$;
        \item $(N,\omega)$ is \emph{$R_\omega$-hyperbolic} if, for each $x \in N$, there exist $c > 0$ and $r > 0$ such that $\mathcal{K}_{(N,\omega)}(v) \geq c|v|$ for every tangent vector $v \in T( \mathbb{B}_N(x,r) )$.
    \end{enumerate}
    Broder, Iliashenko, and Madnick proved the implications `(3) $\Rightarrow$ (2) $\Rightarrow$ (1)' under mild assumptions and asked in \cite[Section 7, Question 3]{broder-iliashenko-madnick-2025-hyperbolicity-and-schwarz-lemmas-in-calibrated-geometry} whether these notions of hyperbolicity coincide when $N$ is a closed calibrated manifold. The following theorem answers this positively.
    \begin{theorem}\label{theorem-characterization-of-hyperbolicity-calibration}
        Let $(N,\omega)$ be a closed $n$-calibrated $m$-manifold. If $(N,\omega)$ is $\omega$-hyperbolic, then there exists $C> 0$ such that
        \begin{align}\label{equation:royden-strong}
            \mathcal{K}_{(N,\omega)}(v)
            \geq
            C^{-1} |v|
            \quad\text{for every $v \in TN$.}
        \end{align}
        In particular, $(N,\omega)$ is $R_\omega$- and $K_\omega$-hyperbolic. Furthermore, if $(N,\omega)$ is not $\omega$-hyperbolic, then $(N,\omega)$ is not $K_\omega$-hyperbolic.
    \end{theorem}

    \begin{remark}
        \Cref{theorem-characterization-of-hyperbolicity-calibration} has been recently proved by Iliashenko, Karigiannis, and Madnick \cite{iliashenko-karigiannis-madnick-2026-montels-theorem-and-tautness-in-calibrated-geometries} and by Broder, Hegarty, and Hudecek (personal communication).
    \end{remark}

    \subsection{Structure of the paper}
    We outline the structure of the paper.
    
    In \Cref{section-preliminaries}, we introduce our fairly standard notation and prerequisites on Sobolev maps.
    
    In \Cref{section-sectional-curvature-bounds-quantitative-geometry}, we formulate our basic setup: that of calibrated manifolds which have a sectional curvature upper bound and an injectivity radius lower bound. We also recall geometric consequences of these assumptions, including small mass isoperimetric inequalities for integral currents and a quantitative Poincaré lemma.

    In \Cref{section-quantitative-geometry-of-qr-curves-under-sectional-curvature-bounds}, we investigate the quantitative geometry of quasiregular curves in our setting. Our results include homological non-triviality of non-constant quasiregular curves $M \to ( N, \omega )$ for closed $M$, modulus of continuity estimates, weak reverse Hölder inequalities, and a new non-linear Caccioppoli type inequality.

    In \Cref{section-equicontinuity-and-sectional-curvature-bounds}, we analyse families of quasiregular curves. In particular, we give equivalent conditions for a family of quasiregular curves to be equicontinuous or normal.

    In \Cref{section-miniowitz-zalcman-rescaling-principle-for-qr-curves}, we prove a version of the Miniowitz--Zalcman rescaling principle using equicontinuity and normality, respectively. Moreover, for the proof of the principle, we propose a notion of Zalcman sequences and study basic properties of such sequences.
    
    In the sections mentioned above, we also treat non-compact targets and formulate corresponding versions of our main results for equibounded families of quasiregular curves.
    
    In \Cref{section-smith-maps-and-hyperbolicity}, we establish the equivalence of the various notions of hyperbolicity for conformal curves for closed targets. In fact, we connect the normality of conformal curves $\mathcal{F}_1(\omega)$ to the uniform boundedness of the derivatives of $\mathcal{F}_1(\omega)$ in a neighbourhood of the origin, similar to Marty's theorem from complex analysis.
    
    In \Cref{section-examples}, we refine the known connections between the quasiregular ellipticity for mappings and for curves by exhibiting a large collection of quasiregular curves factoring e.g. through a Special Lagrangian submanifold of a Calabi--Yau manifold or an associative submanifold of a $G_2$ manifold.

    \subsection{Acknowledgments}
    The first author was supported by Swiss National Science Foundation grant 212867. The second author was supported by the Research Council of Finland, project number 332671 and Center of Excellence FiRST.

    The authors thank Pekka Pankka and Susanna Heikkilä for discussions regarding quasiregular ellipticity and the Miniowitz--Zalcman rescaling principle. We also thank Jesse Madnick for discussions regarding \cite{broder-iliashenko-madnick-2025-hyperbolicity-and-schwarz-lemmas-in-calibrated-geometry}.

\section{Preliminaries}\label{section-preliminaries}

    \subsection{Metric notation}
    Our notation is fairly standard. Open balls and closed balls in \(\R^n\) are denoted by $\mathbb{B}(z,r)$ and $\overline{\mathbb{B}}(z,r)$ respectively. For origin centred balls, we write \(\mathbb{B}(r) = \mathbb{B}(0,r)\), and denote the unit ball \(\mathbb{B}(1)\) by \(\mathbb{B}\). Balls in other spaces are denoted similarly but with a subscript, e.g. \(\mathbb{B}_N(x,r) \subset N\). For a set $A$ in a metric space $X$, the \emph{diameter} is $\diam(A) = \sup\{ d(x,y) \colon x, y \in A \}$.

    A map $f \colon X \to Y$ is \emph{Lipschitz} if
    \begin{align*}
        \LIP(f) = \sup_{ x \neq y } \frac{ d(f(x),f(y)) }{ d(x,y) } < \infty
    \end{align*}
    and $\LIP(f)$ is called the \emph{Lipschitz constant} of $f$. If $L \geq \LIP(f)$, we say that $f$ is $L$-Lipschitz.

    Similarly, a map $f \colon X \to Y$ is \emph{$\beta$-Hölder} if
    \begin{align*}
        \sup_{ x \neq y } \frac{ d(f(x),f(y)) }{ d^{\beta}(x,y) } < \infty.
    \end{align*}

    \subsection{Manifolds, comass norm, and calibrations}
    Our manifolds are smooth, equipped with a Riemannian metric, and we denote the induced distance function by \(d\).
    Submanifolds are equipped with the Riemannian metric induced by the ambient space, and \(M \subset N\) denotes that \(M\) is a submanifold of \(N\). When \(M\) is isometrically embedded in \(N\), we write $M \xhookrightarrow{} N$. If $M$ is oriented, the image of the embedding is equipped with the induced orientation.

    Given a manifold $N$ and a differential $n$-form $\omega$ on \(N\), the \emph{pointwise comass} of $\omega$ is the function $\norm{\omega} \colon N \to \R$ defined by
    \begin{align*}
        \| \omega \|(x)
        =
        \sup\left\{
            \omega( v_1 \wedge \dots \wedge v_n )
            \colon
            \text{$\{v_1,\dots,v_n\}$ is an orthonormal frame on $T_x N$}
        \right\}.
    \end{align*}
    We say that $\omega$ is a \emph{calibration} if $\|\omega\|(x) \leq 1$ for every $x \in N$ and $\omega$ is closed. We say that $\omega$ is \emph{strongly non-vanishing} if there exists $c > 0$ such that $\|\omega\|(x) \geq c$ for every $x \in N$. In case $2 \leq n \leq m$, the pair $(N,\omega)$ is an \emph{$n$-calibrated $m$-manifold} if $N$ is $m$-dimensional and $\omega$ is a strongly non-vanishing calibration of degree $n$. 

    \subsection{Lebesgue spaces, Sobolev maps, the pullback, and the pushforward}
    Let $M$ be a manifold equipped with the (unoriented) Riemannian volume measure. This measure coincides with the $n$-dimensional Hausdorff measure on \(M\) when \(n = \dim M\). The choice of a measure immediately gives a notion of Lebesgue spaces $L^{q}(M)$ and local Lebesgue spaces $L^{q}_{\loc}(M)$. Here $L^{q}_{\loc}(M)$ is defined similarly to $L^{q}(M)$ but $q$-integrability is only required to hold in an open neighbourhood of each point. Similarly, we define $L^{q}_{\loc}( \Omega^{k}(M) )$ for differential forms of degree $k$. For a locally finite measure $\nu$ on $M$ and a $\nu$-measurable set $E \subset M$ with $0 < \nu(E) < \infty$, we denote
    \begin{align*}
        \aint{ E } f \,d\nu \coloneqq \frac{ 1 }{ \nu(E) } \int_{E} f \,d\nu
    \end{align*}
    for every $\nu$-measurable function $f \colon E \to \R$.

    Next, we define Sobolev spaces. Let \(M\) and \(N\) be (Riemannian) manifolds, assume \(N\) is complete, and consider a proper Nash embedding $N \xhookrightarrow{} \R^{\ell}$. For each $p \in [1,\infty)$, the local Sobolev space \(W^{1,p}_{\loc} (M,N)\) is given by    
    \begin{equation*}
        W^{1,p}_{\loc}( M, N ) = \{ f \in W^{1,p}_{\loc}( M, \R^{\ell} ) \mid f(x) \in N \text{ for almost every $x \in M$}  \}.
    \end{equation*}
    This definition is equivalent to several intrinsic ones such as Reshetnyak's approach based on post-composition by Lipschitz maps \cite{reshetnyak-2006-on-the-theory-of-sobolev-classes-of-functions-with-values-in-a-metric-space}; see \cite{convent-van-schaftingen-2016-intrinsic-co-local-weak-derivatives-and-sobolev-spaces-between-manifolds} for further discussion. We only use the global Sobolev space $W^{1,p}(M,N)$ when $M$ has finite volume.

    We consider two notions of energies for Sobolev maps based on the weak differential $Df \colon TM \to TN$ of $f \in W^{1,p}_{\loc}( M, N )$. The definitions of these energies are based on the \emph{operator norm} $\|Df\|$ and the \emph{Hilbert--Schmidt norm} $\|Df\|_{HS}$ respectively. We recall that if $M$ is $n$-dimensional, then $n^{-1/2}\|Df\|_{HS} \leq \|Df\| \leq \|Df\|_{HS}$ almost everywhere. For each precompact open subset $U \subset M$, the associated energies are
    \begin{align*}
        E_p( f|_{U} ) \coloneqq \int_{ U } \|Df\|^p \,dz
        \quad\text{and}\quad
        \mathcal{E}_{p}( f|_{U} ) \coloneqq \int_{U} \|Df\|_{HS}^{p} \,dz.
    \end{align*}
    For the majority of this manuscript, we use the energy \(E = E_n\) for the case \(p=n\) based on the operator norm.

    We denote by $\Omega^{k}_{b}(N)$ the collection of differential $k$-forms whose comass is bounded. Given the weak differential \(Df\) of a map \(f \in W^{1,p}_{\loc}(M,N)\), we define the \emph{pullback} of a smooth and bounded differential form $\eta \in \Omega^{k}_b( N )$ by
    \begin{align*}
        ( f^{*}\eta )( v_1 \wedge \dots \wedge v_k ) = \eta( (Df)(v_1) \wedge \cdots \wedge (Df)(v_k) )
    \end{align*}
    for almost every $v_1 \wedge \dots \wedge v_k \in \bigwedge^k TM$.
    By definition of the comass, clearly $f^{*}\eta \in L^{p/k}_{\loc}( \Omega^{k}(M) )$.

    We consider two variants of Federer--Fleming integral currents: the finite mass integral currents $\mathbf{I}_{*}(N)$ on complete  manifolds and the compactly-supported ones $\mathbf{I}_{*,c}(M)$ on all  manifolds; here $*$ refers to the dimension of the current. We refer to \cite{federer-1969-geometric-measure-theory} for the details on the definitions and only recall the statements that we require.
    
    First, by definition, an integer-rectifiable $k$-current $T$ acts on \emph{bounded} differential $k$-forms $\eta \in \Omega^{k}_b(N)$ in the following manner: there exists a $k$-rectifiable set $E \subset N$ of finite $k$-dimensional Hausdorff measure $\mathcal{H}^{k}$, a measurable multiplicity function $\theta \colon E \to \N = \{1,2,\dots\}$, and a measurable section $v \colon E \to \bigwedge^{k} TN$ of simple unit $k$-vectors satisfying
    \begin{align*}
        M( T ) = \int_{ E } \theta \,d\mathcal{H}^{k} < \infty
        \quad\text{and}\quad
        T( \eta ) = \int_{ E } \eta( v ) \theta \,d\mathcal{H}^{k}
        \quad\text{for every $\eta \in \Omega^{k}_b(N)$}.
    \end{align*}
    We emphasize that the $k$-dimensional Hausdorff measure is normalized so that $\mathcal{H}^k$ coincides with the Lebesgue measure on $\R^k$. The \emph{mass measure} of $T$ is
    \begin{align*}
        \|T\|( B ) = \int_{ B \cap E } \theta \,d\mathcal{H}^k
        \quad\text{for $\mathcal{H}^k$-measurable sets $B \subset N$,}
    \end{align*}
    and $M(T) = \|T\|( N )$ is the \emph{mass} of $T$.

    We note that $T \in \mathbf{I}_{k}(N)$ if $T$ and $\partial T$ are integer-rectifiable currents of dimension $k$ and $k-1$, respectively. Recall that the \emph{boundary} $\partial T$ is defined using a built-in Stokes theorem:
    \begin{align*}
        \partial T( \eta ) = T( d\eta )
        \quad\text{for every $\eta \in \Omega^{k-1}_b(N)$ with $d\eta \in \Omega^{k}_b(N)$},
    \end{align*}
    when $k \geq 1$. Here $\partial T = 0$ if $k = 0$. Clearly $\partial^{2} = 0$. The definition of $\mathbf{I}_{k,c}(N)$ is otherwise the same but we require that $E$ is \emph{precompact}, i.e. the closure $\overline{E} \subset N$ is compact. As $\partial \mathbf{I}_{k,c}(N) \subset \mathbf{I}_{k-1,c}(N)$, we obtain chain complexes $( \mathbf{I}_{*,c}(N), \partial )$ and $( \mathbf{I}_{*}(N), \partial )$.
    
    We note that if $U \subset N$ is a smooth (or Lipschitz) precompact domain on an oriented $n$-dimensional manifold $N$, then equipping $U$ with the orientation of $N$ induces an element $[U] \in \mathbf{I}_{n,c}(N)$ with its \emph{boundary} $\partial [U]$ being $[\partial U] \in \mathbf{I}_{n-1,c}( \partial U ) \subset \mathbf{I}_{n-1,c}( N )$. Here $\partial U$ is equipped with the induced orientation from $U$.

    It is well-known that Federer--Fleming currents can be \emph{pushed forward} by a $\mathcal{C}^1$-regular mapping and, by approximation, by Lipschitz maps. This was recently extended by the first author for Sobolev maps \cite{ikonen-2026-pushforward-of-currents-under-sobolev-maps} using the compatible theory of Ambrosio--Kirchheim integral currents \cite{ambrosio-kirchheim-2000-currents-in-metric-spaces}. We only use these results for $f \in W^{1,n}_{\loc}( M, N )$ when $N$ is a complete (Riemannian) manifold and $M$ is an $n$- or $(n-1)$-dimensional (Riemannian) manifold for $n \geq 2$, and thus we recall a special case of the main results of \cite{ikonen-2026-pushforward-of-currents-under-sobolev-maps}. To set up the statement, for an $n$-integrable Borel function $\rho \colon M \to [0,\infty]$, consider the chain complex $( \mathbf{I}_{*,c}(\rho), \partial )$ defined by
    \begin{align*}
        \mathbf{I}_{*,c}(\rho) 
        &= 
        \bigcup_{ k \geq 1 }\left\{ T \in \mathbf{I}_{k,c}(M) \colon \int_{M} \rho^{k} \,d\|T\| + \int_{M} \rho^{k-1} \,d\|\partial T\| < \infty \right\}
        \\
        &\cup
        \left\{ T \in \mathbf{I}_{0,c}(M) \colon \int_{M} \rho^{0} \,d\|T\|  < \infty \right\}
    \end{align*}
    where the conventions $0^{0} = 0$ and $\infty^{0} = \infty$ are used. Now, the main results of \cite{ikonen-2026-pushforward-of-currents-under-sobolev-maps} prove that, for each $f \in W^{1,n}_{\loc}(M, N )$ and representatives of $f$ and $Df$, there exist a $\rho$ as above and an additive chain map $f_{*} \colon \mathbf{I}_{*,c}(\rho) \to \mathbf{I}_{*}(N)$ with the following properties.
    \begin{enumerate}
        \item the pushforward is compatible with the pullback: if $\eta \in \Omega^{k}_{b}( N )$ and $T \in \mathbf{I}_{k,c}( \rho )$, then
        \begin{align*}
            ( f_* T )( \eta ) = T( f^{*}\eta ) \coloneqq \int_{ E } (f^{*}\eta)(v) \theta \,d\mathcal{H}^{k}
        \end{align*}
        using the representation of $T$ as above.
        \item the boundary and the pushforward commute: if $\eta \in \Omega^{k-1}_{b}(N)$ and $d\eta \in \Omega^{k}_{b}(N)$, then 
        \begin{align*}
            ( f_{*} \partial T )(\eta) = ( f_* T )( d\eta )
        \end{align*}
        for every $T \in \mathbf{I}_{k,c}(\rho)$.
        \item the mass-energy estimate holds: if $T \in \mathbf{I}_{k,c}( \rho )$, then
        \begin{align*}
            M( f_{*}T ) \leq \int_{ M } \|Df\|^{k}\,d\|T\|.
        \end{align*}
        \item if $\eta \in \Omega^{n}_{b}( N )$ and $T \in \mathbf{I}_{n,c}( \rho )$, then
        \begin{align*}
            | (f_*T)( \eta ) | \leq \int_{ M } \|f^{*}\eta\| \,d\|T\| \leq n^{-n/2} \int_{ M } ( \|\eta\| \circ f )\|Df\|^{n}_{HS} \,d\|T\|.
        \end{align*}
    \end{enumerate}
    We give more precise references for these statements. The compatibility with the pullback and the boundary property follow from \cite[Lemmas 5.1 and 5.3]{ikonen-2026-pushforward-of-currents-under-sobolev-maps}, and the main theorem in \cite{ikonen-2026-pushforward-of-currents-under-sobolev-maps} contains the mass-energy estimate in the general case. Regarding the last claim: if $M$ is $(n-1)$-dimensional, then $T = 0$ for $T \in \mathbf{I}_{n,c}( \rho )$ so the conclusion is tautological, and if $M$ is $n$-dimensional, then the mass measure $\|T\|$ is absolutely continuous with respect to the Riemannian volume measure; in this case the first inequality in (4) follows from the compatibility with the pullback and the second inequality from Hadamard's inequality $\|f^{*}\eta\| \leq n^{-n/2}( \|\eta\| \circ f ) \|Df\|^{n}_{HS}$; see e.g. \cite[Lemma 2.1]{ikonen-pankka-2024-liouvilles-theorem-in-calibrated-geometries}.

    The results above are abstract. However, we only apply them in concrete situations. For instance, if $M = \R \times \Sigma$ for a closed oriented manifold $\Sigma$, we have the following elementary consequence of Fubini's theorem: if $\rho \colon M \to [0,\infty]$ is (locally) $n$-integrable, then
    \begin{align*}
        (s,t) \mapsto  \int_{ (s,t) } \int_{ \Sigma } \rho^{n}(r,v) \,d\mathcal{H}^{n-1}(v) \,dr
    \end{align*}
    is (locally) absolutely continuous. In particular, by Hölder's inequality,
    \begin{align*}
        \int_{ \Sigma } \rho^{n-1}(r,v) \,d\mathcal{H}^{n-1}(v)
        \leq
        \left( \mathcal{H}^{n-1}( \Sigma) \right)^{1/n}
        \left(
            \int_{ \Sigma } \rho^{n}(r,v) \,d\mathcal{H}^{n-1}(v)
        \right)^{ \frac{n-1}{n} }
        <
        \infty
    \end{align*}
    for almost every $r \in \R$. Thus, for almost every $(s,t)$, the integral current $[ (s,t) \times \Sigma ]$ belongs to $\mathbf{I}_{n,c}( \rho )$.
    
    Now, if $M$ is an arbitrary oriented $n$-manifold, using ideas similar to those above, given a precompact smooth open set $U \subset M$, we may use a collar neighbourhood of $\partial U$ to approximate $U$ from the inside by smooth domains $( U_j )_j$ such that $[ U_j ] \in \mathbf{I}_{*,c}( \rho )$. Notice that in this case $\int_{ U_j } \rho^{n} \,dz < \infty$ is automatic, so $[U_j] \in \mathbf{I}_{*,c}( \rho )$ only depends on whether $\int_{ \partial U_j } \rho^{n-1} \,d\mathcal{H}^{n-1} < \infty$.

    \subsection{Aspherical calibrations}
    Let $N$ be a complete  manifold. A closed $n$-form \(\omega\) on $N$ is \emph{aspherical} if $\int_{ \mathbb{S}^n } f^{*}\omega = 0$ for every smooth map $f \colon \S^n \to N$, or equivalently $\omega$ annihilates the image of the Hurewicz homomorphism $\pi_n( N ) \rightarrow H_n( N )$. It readily follows that if $f$ is a continuous map in the Sobolev space $W^{1,n}_{\loc}( \mathbb{S}^n, N )$, then $\int_{ \mathbb{S}^{n} } f^{*}\omega = 0$ as well. 
    
    This can be proved using standard arguments: consider a proper Nash embedding $N \xhookrightarrow{} \R^\ell$ and approximate $f$ using convolution and the nearest point projection to $N$. More precisely, we recall the elementary observation that the nearest-point projection of $N$ is well-defined in an open set $U \subset \R^\ell$ containing $N$ and the distance from the compact set $f( \mathbb{S}^n ) \subset N \subset \R^\ell$ to the complement $\R^\ell \setminus U$ is positive. 
    
    Pankka and the second author discussed aspherical calibrations in the context of the bubbling phenomenon \cite{pankka-pim-2025-nodal-resolution-of-quasiregular-curves-via-bubble-trees}. Indeed, the bubbling of quasiregular curves can only occur for calibrations which are \emph{not} aspherical. For example, all exact forms, or, more generally, all forms in the \emph{K\"unneth ideal}, defined in \cite{heikkila-2024-quasiregular-curves-and-cohomology}, are aspherical forms.

\section{Sectional curvature bounds and quantitative geometry}\label{section-sectional-curvature-bounds-quantitative-geometry}

    In this section, we consider manifolds with a lower bound on injectivity radius and an upper bound on sectional curvature. In particular, for every such manifold, there exists a parameter $\delta > 0$ such that the injectivity radius at every point is bounded from below by $2\delta$ and the sectional curvature is bounded from above by $(\pi/(2\delta))^2$ --- notice that a sphere of radius $(2\delta)/\pi$ has precisely these properties. Moreover, every closed manifold satisfies these properties for a sufficiently small $\delta$. This has the advantage of quantifying certain geometric results of interest including the small mass isoperimetric inequalities and a quantitative Poincaré lemma for closed forms on small balls; see the following two subsections.

    For us, the key point of the above geometric assumptions is the following well-known quantitative homotopy lemma. We write $\gamma_{x_0,x} \colon [0,1] \to N$ for a constant-speed geodesic joining $x_0$ to $x$.
    \begin{lemma}\label{lemma-lipschitz-homotopy-bound}
        Let \(N\) be a manifold with an injectivity radius lower bound $2\delta$ and a sectional curvature upper bound $( \pi/(2\delta) )^2$. Let $r \in (0,\delta)$ and $x_0 \in N$. Then the radial homotopy $h \colon [0,1] \times \mathbb{B}_N(x_0,r) \to \mathbb{B}_N( x_0, r)$, where $(t,x) \mapsto \gamma_{x_0,x}(t)$, satisfies the Lipschitz estimate
        \begin{equation*}
            d( h(s,x), h(t,y) ) \leq 2\left( r|t-s| + d(x,y) \right)
        \end{equation*}
        for every $(s,x), (t,y) \in [0,1] \times \mathbb{B}_N(x_0,r)$.
    \end{lemma}
    This follows from standard Jacobi field estimates, see \cite[Section 4]{ikonen-2025-quasiregular-curves-removability-of-singularities} for further discussion.

    \subsection{Small mass isoperimetric inequalities}\label{section-small-mass-isoperimetric-inequalities}
    It is well-known that bounds on injectivity radius and sectional curvature imply that $N$ satisfies small mass isoperimetric inequalities for integral currents \cite{wenger-2007-flat-convergence-for-integral-currents-in-metric-spaces}. We formulate this precisely: we say that a complete manifold $N$ satisfies a \emph{small mass (Euclidean) isoperimetric inequality of dimension $k$, mass $M$, and constant $A$} if 
    \[
        T \in \Set{ S \in  \mathbf{I}_{k}( N ) \mid \partial S = 0\ \text{and}\ M(S) \leq M} \implies \mathrm{FillVol}(T) \leq A ( M(T) )^{ \frac{k+1}{k} }
    \]
    where
    \begin{align*}
        \mathrm{FillVol}(T) = \inf\{ M(S) \mid S \in \partial^{-1}( T ) \}
    \end{align*}
    and $\partial^{-1}( T) = \{ S \in \mathbf{I}_{k+1}(N) \mid \partial S = T \}$.
    The constants depend only on the dimension being considered and $\delta$. More precisely, for a manifold $N$ with injectivity radius at least $2\delta$ and sectional curvature upper bound $(\pi/(2\delta))^2$, the small mass isoperimetric inequality has a mass bound $M = M(\delta,n-1)$ and a constant $A = A(\delta,n-1)$ for dimension $n-1$ and a mass bound $M' = M'(\delta,n)$ and a constant $A' = A'(\delta,n)$ for dimension $n$. It is easy to see that $M(\delta,n-1) < n\omega_{n}( (2\delta)/\pi )^{n-1}$ and $M'(\delta,n) < (n+1)\omega_{n+1}( (2\delta)/\pi )^{n}$ by considering spheres of radius $(2\delta)/\pi$; here $\omega_n$ is the Lebesgue measure of the $n$-dimensional Euclidean ball.

    For the $(n-1)$-dimensional case and $\ell \geq n$, the corresponding constants in Euclidean space $\R^{\ell}$ are $M = \infty$ and 
    \begin{align*}
        A_{n-1} \coloneqq \frac{ \omega_{n} }{ \left( n \omega_n \right)^{\frac{n}{n-1}} } = \frac{ 1 }{ n^{ \frac{n}{n-1} } \omega_n^{ \frac{1}{n-1} } }.
    \end{align*}
    This sharp Euclidean isoperimetric inequality is due to Almgren \cite{almgren-1986-optimal-isoperimetric-inequalities}.
    
    We also define
    \begin{align}\label{equation-small-mass-isoperimetric-inequalities}
        E_N = \min\left\{ M^{ \frac{n}{n-1} } A, \frac{M'}{2} \right\}
    \end{align}
    and
    \begin{align}\label{eq:holder-exponent}
        \beta = \left( \frac{K}{c} \frac{ A }{ A_{n-1} } \right)^{-1} \quad\text{for $K \geq 1$},
    \end{align}
    where $c$ is the comass lower bound of $\omega$.

    The constants \eqref{equation-small-mass-isoperimetric-inequalities} and \eqref{eq:holder-exponent} play a critical role in the coming sections. The former is related to an energy bound below which quasiregular curves become quantitatively $\beta$-Hölder continuous by \cite[Theorem 6.8]{ikonen-2026-pushforward-of-currents-under-sobolev-maps}. See \Cref{theorem-modulus-of-continuity-estimates} for the special case of that theorem in our setting. In light of this, we say that \(N\) has \emph{Hölder data} \(E_N\) and \(\beta\) for $K \geq 1$.

    \subsection{Quantitative Poincaré lemma}\label{section-quantitative-poincare-lemma}

    The following lemma shows that closed forms are exact in small enough balls, with quantitative estimates for the comass of the potential.
    \begin{lemma}\label{lemma-existence-of-potentials}
     Let \(N\) be a manifold with an injectivity radius lower bound $2\delta$ and a sectional curvature upper bound $( \pi/(2\delta) )^2$. Let $\omega$ be a closed smooth $n$-form on $\mathbb{B}_N( x_0, R ) \subset N$ and let $0 < r \leq \min\{R,\delta\}$. Then $\omega$ has a potential $\tau$ on $\mathbb{B}_N( x_0, r )$ whose comass is bounded from above by $2^{n-1}r \sup_{ x\in \mathbb{B}_N(x_0,r) } \|\omega\|(x)$.
    \end{lemma}
    \begin{proof}
        We let $h \colon [0,1] \times \mathbb{B}_{N}( x_0, r ) \to \mathbb{B}_{N}( x_0, r )$ be the radial homotopy to $x_0$. For every $t \in [0,1]$, we define $h_t(x) \coloneqq h(t,x)$, and let $X_t(x) \in T_{ h_t(x) } N$ denote the tangential derivative of $s \mapsto h_s(x)$ at $s = t$ for each $(t,x) \in (0,1) \times \mathbb{B}_{N}(x_0, r)$. The smoothness (even at $t = 0$) of $(t,x) \mapsto X_t$ readily follows from normal coordinates. By Cartan's magic formula and the definition of the Lie derivative, we have
        \begin{align*}
            h_t^{*} d( \iota_{ X_t }\omega ) = h_{t}^{*}(\mathcal{L}_{ X_t }\omega) = \frac{ d }{ dt } h_t^{*}\omega,
        \end{align*}
        where $\mathcal{L}_{ X_t }$ and $\iota_{ X_t }$ denote the Lie derivative and interior product along $X_t$, respectively. Thus, denoting
        \begin{align*}
            \tau
            \coloneqq
            \int_{ 0 }^{ 1 } h_{t}^{*}\left( \iota_{X_t}\omega \right) \,dt,
        \end{align*}
        we obtain a potential for $\omega$. Recalling $|X_t(x)| \leq r$ by definition, the comass estimate for $\tau$ follows from \Cref{lemma-lipschitz-homotopy-bound} and the comass upper bound for $\omega$.
    \end{proof}

\section{Quasiregular curves and sectional curvature}\label{section-quantitative-geometry-of-qr-curves-under-sectional-curvature-bounds}

    \subsection{Homological non-triviality of quasiregular curves}\label{section-homological-non-triviality}
    The following result describes two important properties of quasiregular curves defined on a closed manifold.
    \begin{lemma}\label{lemma-homological-positive-current}
    Let \(M\) be a closed, connected and oriented $n$-manifold, and let $(N,\omega)$ be an $n$-calibrated $m$-manifold. Then each quasiregular curve $F \colon M \to ( N, \omega )$ satisfies the following:
    \begin{enumerate}
        \item the integral homology class of $F_{*}[M]$ in $H_{n}( N, \mathbb{Z} )$ is non-trivial if and only if $F$ is non-constant.
        \item the current $F_{*}[M]$ is a positive $\omega$-current.
    \end{enumerate}
    \end{lemma}
    Recall that an integer-rectifiable $n$-current $T$ on $N$ has a representation in terms of an $n$-rectifiable set $E \subset N$, a measurable section of simple unit $n$-vectors $v \colon E \to \bigwedge^{n} TN$, and a multiplicity function $\theta \colon E \to \{1,2,\dots\}$. With this in mind, following Harvey--Lawson \cite[Definition 1.4]{harvey-lawson-1982-calibrated-geometries}, we say that $T$ is a \emph{positive $\omega$-current} if $\omega(v) \geq 0$ $\mathcal{H}^n$-almost everywhere on $E$. We are now ready to prove the result.
    \begin{proof}
        Let $K$ be the distortion constant of $F$ and let $c$ be a lower bound for the comass of $\omega$. Then
        \begin{align*}
            \frac{c}{K} E( F ) \leq \int_{M} F^{*}\omega = (F_{*}[M])( \omega )
        \end{align*}
        by the distortion inequality and the compatibility of the pullback and the pushforward. Now, regarding (1), if we had an integral $(n+1)$-current $\Sigma$ with $\partial \Sigma = F_{*}[M]$, it would follow that
        \begin{align*}
            (F_{*}[M])(\omega) = (\partial \Sigma)( \omega ) = \Sigma( d\omega ) = 0.
        \end{align*}
        Hence $F_{*}[M]$ being exact implies that $F$ has zero energy, i.e. that $F$ is constant. In the case that $F$ is constant, obviously $F_{*}[M]$ is the zero current.

        Regarding (2), by the compatibility of the pushforward and the pullback, every smooth and bounded function $\phi \colon N \to [0,\infty)$ satisfies
        \begin{align*}
            0 \leq \int_{M} ( \phi \circ F ) \|DF\|^n \,dz \leq \frac{K}{c} \int_{M} ( \phi \circ F ) F^{*}\omega = (F_{*}[M])( \phi \omega ).
        \end{align*}
        Since this holds for every smooth $\phi$, by approximation, the conclusion also holds for every bounded Borel function $\phi \colon N \to [0,\infty)$. The current being $\omega$-positive is immediate.
    \end{proof}
    Combining the above with the small mass isoperimetric inequalities for integral currents, we obtain the following version of the energy gap \cite[Theorem~4.2]{pankka-pim-2025-nodal-resolution-of-quasiregular-curves-via-bubble-trees}.
    \begin{corollary}\label{lemma-homological-non-triviality}
         Let \(M\) be a closed oriented \(n\)-manifold, and let \((N,\omega)\) be an \(n\)-calibrated \(m\)-manifold with an injectivity radius lower bound $2\delta$ and a sectional curvature upper bound $( \pi/(2\delta) )^2$. Let \(M'\) be the mass bound for the \(n\)-dimensional isoperimetric inequality of \(N\). Then \(E(F) \geq  M'\) for every non-constant quasiregular curve $F \colon M \to ( N, \omega )$.
    \end{corollary}
    \begin{proof}
        In case $E(F) < M'$, it holds that $M( F_{*}[M] ) \leq E(F) < M'$, and thus there exists an integral $(n+1)$-current $\Sigma$ with $\partial \Sigma = F_{*}[M]$. Then $F$ being constant follows from \Cref{lemma-homological-positive-current}.
    \end{proof}

    As a consequence of the removability of singularities theorem \cite[Theorem 1.9]{ikonen-2025-quasiregular-curves-removability-of-singularities}, we obtain the following corollary for entire quasiregular curves.
    \begin{corollary}\label{corollary-homological-non-triviality}
        Let \((N,\omega)\) be an \(n\)-calibrated  \(m\)-manifold with an injectivity radius lower bound $2\delta$ and a sectional curvature upper bound $( \pi/(2\delta) )^2$. Let \(M'\) be the mass bound for the \(n\)-dimensional isoperimetric inequality of \(N\). Then \(E(F) \geq  M'\) for every non-constant quasiregular curve $F \colon \R^n \to ( N, \omega )$.
    \end{corollary}
    \begin{remark}\label{remark-homological-non-triviality-asherical}
        The same argument shows that when \(\omega\) is aspherical, a Liouville-type theorem holds: \emph{an entire finite energy quasiregular curve \(\R^n \to (N,\omega)\) is constant}.
    \end{remark}    

\subsection{Modulus of continuity estimates}

    We have the following special case of \cite[Theorem 6.8]{ikonen-2026-pushforward-of-currents-under-sobolev-maps}, which can be understood as a Schwarz--Ahlfors--Pick theorem for quasiregular curves.
\begin{theorem}\label{theorem-modulus-of-continuity-estimates}
    Let \((N,\omega)\) be an \(n\)-calibrated  \(m\)-manifold with an injectivity radius lower bound $2\delta$, a sectional curvature upper bound $( \pi/(2\delta) )^2$, and H\"older data \(E_N\) and \(\beta\) for $K \geq 1$.
    
    Then there exists \(C(n,\beta) > 0\) with the following property: for any ball $\mathbb{B}(x_0,2r) \subset \R^n$ and every $K$-quasiregular curve $F \colon \mathbb{B}(x_0,2r) \to ( N, \omega )$ with $E( F|_{ \mathbb{B}(x_0,2r) } ) \leq E_N$, it holds that
    \begin{align*}
        d( F(x), F(y) ) 
        \leq 
        C(n,\beta) \left( E( F|_{ \mathbb{B}(x_0,2r) } ) \right)^{1/n} \left( \frac{ |x-y| }{ 2r } \right)^{ \beta }
    \end{align*}
    whenever $x, y \in \mathbb{B}( x_0, r )$. In particular,
    \begin{align*}
        \diam( F( \mathbb{B}(x_0,r) ) ) \leq C(n,\beta) \left( E( F|_{ \mathbb{B}(x_0,2r) } ) \right)^{1/n}.
    \end{align*}
\end{theorem}

\subsection{Weak reverse Hölder inequality}
    We use the geometric assumptions to derive a quantitative weak reverse Hölder inequality in our setting. See \cite{onninen-pankka-quasiregular-curves-holder-continuity-and-higher-integrability,heikkila-2023-signed-quasiregular-curves} for a related results. The following statement is implicitly contained in the proof of \cite[Theorem 6.8]{ikonen-2026-pushforward-of-currents-under-sobolev-maps}.
\begin{proposition}\label{proposition-quasiminimality-of-energy-estimate}
    Let \((N,\omega)\) be an \(n\)-calibrated  \(m\)-manifold with an injectivity radius lower bound $2\delta$, a sectional curvature upper bound $( \pi/(2\delta) )^2$, and H\"older data \(E_N\) and \(\beta\) for $K \geq 1$. Then for any ball $\mathbb{B}(x_0,r) \subset \R^n$ and every $K$-quasiregular curve $F \colon \mathbb{B}(x_0,r) \to ( N, \omega )$ with $E( F|_{ \mathbb{B}(x_0,r) } )  \leq E_N$, it holds that 
    \begin{align*}
        \aint{ \mathbb{B}(x_0,s) } \|DF\|^{n} \,dz 
        \leq
        \frac{ 1 }{ \beta }
        \left(
            \aint{ \partial \mathbb{B}(x_0,s) }  \|DF\|^{n-1}(z)\,d\mathcal{H}^{n-1}(z)
        \right)^{ \frac{n}{n-1}  }
    \end{align*}
    for almost every $s \in (0,r)$.
\end{proposition}

\begin{proof}
    The proof is similar to that of \cite[Theorem 6.8]{ikonen-2026-pushforward-of-currents-under-sobolev-maps}: indeed, using the notation from that proof, we let
    \begin{align*}
        H(s) = E( F|_{ \mathbb{B}(x_0,s) } )
        \quad\text{and}\quad
        h(s) = \int_{ \partial \mathbb{B}(x_0,s) }  \|DF\|^{n-1}(z)\,d\mathcal{H}^{n-1}(z),
    \end{align*}
    and note that \cite[equation (6.8)]{ikonen-2026-pushforward-of-currents-under-sobolev-maps} shows
    \begin{align*}
        \frac{ H(s) }{ \omega_n s^{n} }
        &\leq 
        \frac{ K }{ c } A \frac{ \left( n \omega_n s^{n-1} \right)^{ \frac{n}{n-1} } }{ \omega_n s^{n} }  \left( \frac{ h(s) }{ n \omega_n s^{n-1}  } \right)^{ \frac{n}{n-1} }
        =
        \frac{ 1 }{ \beta } \left( \frac{ h(s) }{ n \omega_n s^{n-1}  } \right)^{ \frac{n}{n-1} }
    \end{align*}
    for almost every $s \in (0,r)$. This is exactly the claim.
\end{proof}
    \Cref{proposition-quasiminimality-of-energy-estimate} leads to a weak reverse Hölder inequality by standard arguments.
\begin{proposition}\label{proposition-weak-reverse-holder-inequality}
    Let \((N,\omega)\) be an \(n\)-calibrated  \(m\)-manifold with an injectivity radius lower bound $2\delta$, a sectional curvature upper bound $( \pi/(2\delta) )^2$, and H\"older data \(E_N\) and \(\beta\) for $K \geq 1$.
    
    Then there exists $p = p( n, \beta ) > n$ and $C = C( n, \beta ) > 0$ with the following properties:  for any ball $\mathbb{B}(x_0,3r) \subset \R^n$ and every $K$-quasiregular curve $F \colon \mathbb{B}(x_0,3r) \to ( N, \omega )$ with $E( F|_{ \mathbb{B}(x_0,3r) } ) \leq E_N$, it holds that
    \begin{align*}
        \left( \aint{ \mathbb{B}(x_0,r) } \|DF\|^p \,dz \right)^{1/p}
        \leq
        C\left( \aint{ \mathbb{B}(x_0,2r) } \|DF\|^n \,dz \right)^{1/n}.
    \end{align*}
\end{proposition}
\begin{proof}
    If we fix $\rho \in (0,r)$ and $a \in \mathbb{B}(x_0,r)$, by Fubini's theorem, there exists $s \in (\rho,2\rho)$ satisfying the conclusion of \Cref{proposition-quasiminimality-of-energy-estimate} and 
    \begin{align*}
        \int_{ \partial \mathbb{B}(a,s) } \|DF\|^{n-1} \,d\mathcal{H}^{n-1}
        \leq
        \frac{ 1 }{ \rho }
        \int_{ \mathbb{B}(a,2\rho) \setminus \overline{\mathbb{B}}(a,\rho) }  \|DF\|^{n-1} \,dz
        \leq
        \frac{ 1 }{ \rho }
        \int_{ \mathbb{B}(a,2\rho) }  \|DF\|^{n-1} \,dz.
    \end{align*} 
    This implies
    \begin{align*}
        \aint{ \partial \mathbb{B}(a,s) } \|DF\|^{n-1} \,d\mathcal{H}^{n-1}
        \leq
        \frac{ 2^{n} }{ n } 
        \aint{ \mathbb{B}(a,2\rho) }  \|DF\|^{n-1}\,dz.
    \end{align*}
    We also notice that
    \begin{align*}
        \aint{ \mathbb{B}(a,\rho) } \|DF\|^{n} \,dz
        \leq
        2^{n}
        \aint{ \mathbb{B}(a,s) } \|DF\|^{n} \,dz.
    \end{align*}
    Therefore, by the conclusion of \Cref{proposition-quasiminimality-of-energy-estimate} and the above, it holds that
    \begin{align*}
        \aint{ \mathbb{B}(a,\rho) } \|DF\|^n \,dz
        &\leq
        \frac{ 2^{n} }{ \beta }
            \left( \aint{ \partial \mathbb{B}(a,s) }  \|DF\|^{n-1}\,d\mathcal{H}^{n-1} \right)^{ \frac{n}{n-1} }
        \\
        &\leq
        \frac{ 2^{n} }{ \beta }
        \left( \frac{ 2^{n} }{ n } \right)^{ \frac{n}{n-1}  }
        \left( \aint{ \mathbb{B}(a,2\rho) }  \|DF\|^{n-1}\,dz \right)^{ \frac{n}{n-1} }.
    \end{align*}
    It follows that
    \begin{align*}
        \left( \aint{ \mathbb{B}(a,\rho) } \|DF\|^{n}(z) \,dz \right)^{\frac{n-1}{n}}
        \leq
        \frac{ 2^{n-1} }{ \beta^{ \frac{n-1}{n} } } 
        \frac{ 2^{n} }{ n }
        \aint{ \mathbb{B}(a,2\rho) }  \|DF\|^{n-1}\,dz.
    \end{align*}
    By standard iteration arguments (see Gehring's lemma \cite[Theorem 3.22]{bjorn-bjorn-2011-non-linear-potential-theory-on-metric-spaces}), there exist $p = p(n, \beta ) > n$ and $C = C(n, \beta ) > 0$ such that
    \begin{align*}
        \left( \aint{ \mathbb{B}(x_0,r) } \|DF\|^{p}(z) \,dz \right)^{1/p}
        \leq
        C\left( \aint{ \mathbb{B}(x_0,2r) } \|DF\|^{n}(z) \,dz \right)^{1/n}.
    \end{align*}
    The claim follows.
\end{proof}

\subsection{Caccioppoli's inequality}\label{section-caccioppoli-inequality}
In this section, we prove the following non-linear version of a Caccioppoli type inequality. A similar inequality is known for Euclidean targets \cite{onninen-pankka-quasiregular-curves-holder-continuity-and-higher-integrability}.
\begin{proposition}\label{proposition-caccioppoli-inequality}
    Let \((N,\omega)\) be an \(n\)-calibrated  \(m\)-manifold with an injectivity radius lower bound $2\delta$, a sectional curvature upper bound $( \pi/(2\delta) )^2$, and comass lower bound \(c > 0\), and let $K \geq 1$.

    Then there exists $C = C(n) > 0$ with the following property: for any ball $\mathbb{B}(x_0,2r) \subset \R^n$ and every $K$-quasiregular curve $F \colon \mathbb{B}(x_0,2r) \to ( N, \omega )$ with $$\diam( F( \mathbb{B}(x_0,2r) ) ) < \delta,$$ it holds that
    \begin{align*}
        E( F|_{ \mathbb{B}(x_0,r) } ) 
        \leq 
        C(n)
        \left( 
            \frac{ K }{ c }
        \right)^n
        \diam( F( \mathbb{B}(x_0,2r) ) )^{n}.
    \end{align*}
\end{proposition}

\begin{proof}
    Let $y_0 \in F( \mathbb{B}(x_0,2r) )$. Let \(R\) satisfy $\diam( F( \mathbb{B}( x_0, 2r ) ) ) < R < \delta$, and let $\tau$ be a potential of $\omega$ in $\mathbb{B}( y_0, R )$ provided by \Cref{lemma-existence-of-potentials}. Then \cite[Proposition 3.1]{pankka-quasiregular-curves} implies, for some $C = C(n) > 0$, that
    \begin{align*}
        \frac{c}{K} 
        E( F|_{ \mathbb{B}(x_0,r) } ) 
        &\leq
        C K^{n-1}
        \left( \frac{2}{r} \right)^n
        \int_{ \mathbb{B}(x_0,2r) } \left( \frac{ \|\tau\|^{n} }{ \|\omega\|^{n-1} } \right) \circ F \,dz
        \\
        &\leq
        C K^{n-1}
        2^{2n}\omega_n
        \frac{ 2^{n(n-1)}R^{n} }{ c^{n-1} }.
    \end{align*}
    Rearranging the inequality leads to
    \begin{align*}
        E( F|_{ \mathbb{B}(x_0,r) } ) 
        &\leq
        C \omega_n 2^{ n(n+1) } \left( \frac{K}{c} \right)^n R^n,
    \end{align*}
    and passing to the limit $R \rightarrow \diam( F( \mathbb{B}( x_0, 2r ) ) )$ completes the proof.
\end{proof}

\section{Equicontinuity and sectional curvature upper bounds}\label{section-equicontinuity-and-sectional-curvature-bounds}

    In this section, we consider various equicontinuity and normality criteria for families of quasiregular curves. Here and below, a family $\mathcal{F}$ of maps $f \colon X \to Y$ between metric spaces is \emph{equicontinuous} if, for every $\eta > 0$ and $x \in X$, there exists $r > 0$ such that $\diam f( \mathbb{B}( x,r ) ) < \eta$ for every $f \in \mathcal{F}$. Recall also that a subset $U \subset \Omega \subset \R^n$ is \emph{precompact} if the closure $\overline{U}$ in $\R^n$ is a compact subset in $\Omega$.
\begin{theorem}\label{theorem-equicontinuous-family}
    Let $\Omega \subset \R^n$ be an open set, and let \((N,\omega)\) be an \(n\)-calibrated \(m\)-manifold with an injectivity radius lower bound $2\delta$, a sectional curvature upper bound $( \pi/(2\delta) )^2$, and H\"older data \(E_N\) and \(\beta\) for $K \geq 1$. Let $\mathcal{F}$ be a family of $K$-quasiregular curves $\Omega \to ( N, \omega )$.
    Then the following are equivalent: 
    \begin{enumerate}
        \item $\mathcal{F}$ is equicontinuous;  
        \item for every compact $E \subset \Omega$, there exists a constant $M > 0$ such that
        \begin{align*}
            d( F(x), F(y) ) \leq M |x-y|^{ \beta }
        \end{align*}
        for every $x, y \in E$ and $F \in \mathcal{F}$;
        \item for every $x \in \Omega$, there exists $r > 0$ such that $ \sup_{F \in \mathcal{F}} E( F|_{ \mathbb{B}(x,r) } ) \leq E_N$;
        \item for every $\eta > 0$ and $x \in \Omega$, there exists $r > 0$ such that 
        \[
            \sup_{F \in \mathcal{F}} E( F|_{ \mathbb{B}(x,r) } ) < \eta;
        \]
        \item the family of Radon measures $\{ \star F^{*}\omega \mid F \in \mathcal{F} \}$ is weak-\(\star\) precompact and its weak-\(\star\) closure contains only non-atomic Radon measures; and
        \item for the exponent $p$ from \Cref{proposition-weak-reverse-holder-inequality}, the family of functions $\{ \star F^{*}\omega \mid F \in \mathcal{F} \}$ is weakly compact in $L^{p/n}_{\loc}( \Omega )$.
    \end{enumerate}
\end{theorem}
\begin{remark}\label{remark-on-equicontinuity}
    In the statement above, the space of Radon measures is understood as elements in the topological dual of the compactly-supported continuous functions $\mathcal{C}_{c}(\Omega)$. Here $\mathcal{C}_{c}(\Omega)$ is equipped with the following (inductive) topology: $f_j \rightarrow f$ if (and only if) there exists a compact set $K \subset \Omega$ such that all but finitely many of $(f_j)_j$ are supported in $K$ and $f_j \rightarrow f$ uniformly. 

    Weak compactness of the Lebesgue space $L^{p/n}_{\loc}( \Omega )$ is understood as follows: $\mathcal A \subset L^{p/n}_{\loc}( \Omega )$ is weakly compact if for every sequence $( h_j )_j \subset \mathcal A$, there exists $h \in L^{p/n}_{\loc}( \Omega )$ and a subsequence $( h_{j_k} )_k$ such that $h_{j_k} \rightharpoonup h$ weakly in $L^{p/n}( K )$ for every compact $K \subset \Omega$.

    In the last two conditions of \Cref{theorem-equicontinuous-family}, an equivalent condition is obtained if $\star F^{*}\omega$ is replaced by $\|DF\|^n$ or $\|DF\|^{n}_{HS}$. This follows from the two-sided pointwise comparability of all these functions for quasiregular curves.
\end{remark}
\begin{proof}
    We first prove the equivalence of (1), (2), and (4). 
    Observe that `(2) $\Rightarrow$ (1)' is immediate and `(1) $\Rightarrow$ (4)' follows from \Cref{proposition-caccioppoli-inequality}. It remains to prove `(4) $\Rightarrow$ (2)'. To see this, let $0 < \eta \leq E_N$. Consider $x \in \Omega$ and let $ r > 0$ be sufficiently small that $\mathbb{B}(x,2r) \subset \Omega$ and
    \begin{align*}
        \sup_{F \in \mathcal{F}} E( F|_{ \mathbb{B}(x,2r) } ) < \eta.
    \end{align*}
    Then \Cref{theorem-modulus-of-continuity-estimates} implies
    \begin{align*}
        d( F(y), F(z) ) \leq C(n,\beta) \eta^{1/n} \left( \frac{ |y-z| }{ 2r } \right)^{\beta} 
    \end{align*}
    for every $y, z \in \mathbb{B}( x, r )$ and $F \in \mathcal{F}$. Now a standard covering argument yields (2). 
    Now, clearly `(4) $\Rightarrow$ (3)' holds, while `(3) $\Rightarrow$ (1)' follows from \Cref{theorem-modulus-of-continuity-estimates}. This establishes that (1)--(4) are equivalent.

    It remains to show that (5) and (6) are equivalent to (1)--(4). We first prove `(5) $\Rightarrow$ (3)'. Suppose for contradiction that there exists \(x \in \Omega\) such that 
    \[
    \sup_{F \in \mathcal{F}} E(F|_{\mathbb{B}(x,r)}) > E_N \quad\text{for all sufficiently small \(r>0\)}.
    \]
    Hence, there is a sequence \((F_j)_j\) in \(\mathcal{F}\) which satisfies \(\limsup_{j \to \infty} E(F_j|_{\mathbb{B}(x,r)}) \geq E_N\) for every \(r>0\) which is sufficiently small.
    Set \(\mu_j \coloneqq \star F_j^*\omega\). Without loss of generality and by (5), we may assume that \((\mu_j)_j\) converges in the weak-\(\star\) topology to a non-atomic Radon measure \(\mu\). Hence
    \[
        \frac{c}{K} E_N \leq \frac{c}{K} \limsup_{j\to\infty} E(F_j|_{\overline{\mathbb{B}}(x,r)}) \leq \limsup_{j\to\infty} \mu_j(\overline{\mathbb{B}}(x,r)) \leq \mu (\overline{\mathbb{B}}(x,r)),
    \]
    for the comass lower bound $c > 0$ of $\omega$. Letting \(r \to 0\), we obtain \(0 < (c/K)E_N \leq \mu(\{x\})\). This is a contradiction as \(\mu\) is non-atomic.

    Clearly `(6) $\Rightarrow$ (5)'. This follows from the fact that each compactly-supported continuous function in $\Omega$ is in \(L^q(D)\) for every \(q \in (1,\infty)\) and all precompact open sets \(D\subset \Omega\).
    Then it remains to show `(3) $\Rightarrow$ (6)' to establish the full equivalence. By (3), for each \(x \in \Omega\), there exists \(r> 0\) for which \(\mathbb{B}(x,3r) \subset \Omega\) and \(\sup_{F\in\mathcal{F}} E(F|_{\mathbb{B}(x,3r)}) \leq E_N\). 
    Now \Cref{proposition-weak-reverse-holder-inequality} implies that
    \begin{equation}\label{eq:thm-equicont-family-bounded-from-weak-reverse-holder}
        \sup_{F \in \mathcal{F}} \int_{\mathbb{B}(x,r)} \norm{DF}^p \,dz 
        \leq (C/2)^p (\omega_n r^n)^{1 - p/n} E_N^{p/n}.    
    \end{equation}    
    From this, we conclude that the family \(\set{ \norm{DF}^n \mid F \in \mathcal{F}}\) is bounded in \(L^{p/n}(D)\) for each precompact open set \(D \subset \Omega\) since each such \(D\) can be covered with finitely many balls satisfying \eqref{eq:thm-equicont-family-bounded-from-weak-reverse-holder}. The claim follows from an exhaustion by compact sets of \(\Omega\) and a standard diagonal argument.
\end{proof}

As a corollary of our equicontinuity criteria \Cref{theorem-equicontinuous-family} and \Cref{proposition-weak-reverse-holder-inequality}, we have the following result concerning the stability of quasiregularity and the continuity of certain associated measures under local uniform limits.
\begin{corollary}\label{corollary-convergence-of-pullbacks}
    Let $\Omega \subset \R^n$ be an open set, and let \((N,\omega)\) be an \(n\)-calibrated  \(m\)-manifold with an injectivity radius lower bound $2\delta$, a sectional curvature upper bound $( \pi/(2\delta) )^2$, and $K \geq 1$.

    If $( F_j \colon \Omega \to (N,\omega ) )_j$ is a sequence of $K$-quasiregular curves converging to $F \colon \Omega \to (N,\omega)$ locally uniformly, then \(F\) is a \(K\)-quasiregular curve and in each precompact open $U \subset \Omega$, it holds that
    \begin{align*}
        E( F|_{U} ) \leq \liminf_{ j \to \infty } E( F_j|_{U} ) \leq \sup_{ j } E( F_j|_{U} ) < \infty.
    \end{align*} 
    Moreover, for any $\eta \in \Omega_{b}^{k}( N )$ and the exponent $p > n$ from \Cref{proposition-weak-reverse-holder-inequality}, $( F_j^{*}\eta|_{U} )_j$ weakly converges to $F^{*}\eta|_{U}$ in $L^{p/k}( \Omega^{k}(U) )$. In particular, 
     \[
         \lim_{j\to\infty} \int_V F_j^{*}\omega = \int_V F^{*}\omega
     \]
     for all precompact Borel sets \(V \subset \Omega\).
\end{corollary}
    The energy bounds and the quasiregularity of $F$ were already established in \cite[Theorem 4.1]{pankka-quasiregular-curves}. We provide an independent proof as a simple application of \Cref{proposition-weak-reverse-holder-inequality} and \Cref{theorem-equicontinuous-family}.
\begin{proof}
    The energy bounds for the sequence $( F_j )_j$ and $F$ are immediate from \Cref{theorem-equicontinuous-family}.
    When $U \subset \Omega$ is precompact and open, \Cref{proposition-weak-reverse-holder-inequality} and \Cref{theorem-equicontinuous-family} imply that $( F_j )_j$ is bounded in $W^{1,p}( U, N )$, so the weak convergence of $( F_j )_j$ to $F$ in $W^{1,p}( U, N )$ is immediate. Using this fact, the claimed weak convergence $( F_j^{*}\eta|_{U} )_j$ to $F^{*}\eta|_{U}$ follows using standard arguments. We give an outline of the argument below.

    We reduce the argument to a simpler case as in \cite[Proof of Lemma 5.1]{ikonen-2026-pushforward-of-currents-under-sobolev-maps}. We start by recalling that there exists a $1$-Lipschitz proper map $g \colon N \to \R$. Then, let $\psi_j = h( 2^{-j}g )$ for a smooth $2$-Lipschitz map $h \colon \R \to [0,1]$ that satisfies $h(r) = 1$ if $r \leq 1$ and $h(r) = 0$ if $r \geq 2$. Here $\psi_j$ is compactly-supported $2^{1-j}$-Lipschitz and $( \psi_j )_j$ converges locally uniformly to the constant map one. Now, the claim about the weak convergence of pullbacks follows if it holds for each $\eta_j = \psi_j \eta$, by construction. This reduces the claim to the case where $\eta$ is compactly-supported. Since every compactly-supported $\eta$ is a finite sum of smooth forms compactly-supported in a coordinate ball, we lose no generality in supposing that the support of $\eta$ lies in a coordinate ball. These facts reduce the statement to the special case where $N = \R^\ell$ for some $\ell \geq n$ and $\eta$ is a smooth and compactly-supported $k$-form in $\R^\ell$. Now, the claim follows from standard weak continuity properties of minors of the weak differentials of $( F_j )_j$ \cite[Lemma 5.10]{rindler-2018-calculus-of-variations}. By applying this result for $\omega$, we obtain the $K$-quasiregularity of $F$ as well.
\end{proof}

With the above results on equicontinuity and stability of quasiregularity, we present a normality criterion for families of quasiregular curves. Before that, we recall some standard terminology. Here and below, a family $\mathcal{F}$ of maps $f \colon X \to Y$ is \emph{equibounded in space} if for every compact set $E \subset X$, there exists a compact set $E' \subset Y$ such that $f(E) \subset E'$ for every $f \in \mathcal{F}$. As is standard, an \emph{equibounded family} refers to any family that is equibounded in space. If $Y$ is compact, this holds for any family. Finally, a family $\mathcal{F}$ is \emph{normal} if every sequence has a subsequence which converges locally uniformly. We recall that by the Arzelà--Ascoli theorem, when $X$ and $Y$ are separable and locally compact metric spaces, normality is equivalent to the family being equicontinuous and equibounded in space. 

\begin{corollary}\label{corollary-normal-family}
    Let $\Omega \subset \R^n$ be an open set, and let \((N,\omega)\) be an \(n\)-calibrated \(m\)-manifold with an injectivity radius lower bound $2\delta$, a sectional curvature upper bound $( \pi/(2\delta) )^2$, and H\"older data \(E_N\) and \(\beta\) for $K \geq 1$. Let $\mathcal{F}$ be an equibounded family of $K$-quasiregular curves $\Omega \to ( N, \omega )$.   
    Then the following are equivalent:
    \begin{enumerate}
        \item $\mathcal{F}$ is normal;
        \item for every $x \in \Omega$, there exists $r > 0$ such that $\sup_{F \in \mathcal{F}} E( F|_{ \mathbb{B}(x,r) } ) \leq E_N$;
        \item for every compact $E \subset \Omega$, there exists a constant $M > 0$ such that
    \begin{align*}
         d( F(x), F(y) ) \leq M |x-y|^{ \beta }
    \end{align*}
    for every $x, y \in E$ and every $F \in \mathcal{F}$; and 
    \item the family of Radon measures $\{ \star F^{*}\omega \mid F \in \mathcal{F} \}$ is weak-\(\star\) precompact and its weak-\(\star\) closure contains only non-atomic Radon measures.
    \end{enumerate}
\end{corollary}

    \begin{remark}
        When an equibounded family $\mathcal{F}$ of $K$-quasiregular curves is not equicontinuous,  condition (4) of \Cref{corollary-normal-family} can fail in two ways. Either the family $\{ \star F^{*}\omega \mid F \in \mathcal{F} \}$ is not weak-$\star$ precompact, or it is weak-$\star$ precompact but its weak-$\star$ closure contains atomic measures. The latter situation is exactly the bubbling phenomenon analysed by Pankka and the second author in \cite{pankka-pim-2025-nodal-resolution-of-quasiregular-curves-via-bubble-trees}.
        Our rescaling principle, from the following section, can be applied in both situations.
    \end{remark}

    \begin{remark}[Pointwise equicontinuity, equiboundedness and normality]
    We say that a family of mappings $\mathcal{F} = \{ f \colon X \to Y \}$ is \emph{equicontinuous (or equibounded or normal respectively) at \(x \in X\)} if there exists an open neighbourhood \(U\) of \(x\) such that the restricted family \(\mathcal{F}|_U \coloneqq \Set{f|_U \mid f \in \mathcal{F}}\) is equicontinuous (or equibounded or normal respectively). Note that \(\mathcal{F}\) being equibounded at \(x\) is equivalent to there existing a neighbourhood \(U\) of \(x\) for which \(\mathcal{F}|_U\) is a bounded family.
    
    The results in this and the following section are stated for ``globally'' equicontinuous (or equibounded or normal) families. However, these results are essentially local in the sense that pointwise versions may be readily obtained by restricting the family to a neighbourhood of a point where those aforementioned properties hold.
    \end{remark}

\section{Miniowitz--Zalcman rescaling principles for quasiregular curves}\label{section-miniowitz-zalcman-rescaling-principle-for-qr-curves}

In this section, we prove \Cref{theorem-rescaling-principle,theorem-qr-ellipticity-equivalent-to-normality} as consequences of our rescaling principles for non-equicontinuous and non-normal families of quasiregular curves.
Though the standard proof of the Miniowitz--Zalcman rescaling principle, based on rescaling at the points where the Hölder semi-norm blows up (see for example \cite[Lemma~1]{miniowitz-1982-normal-families-of-quasimeromorphic-mappings} or \cite[Theorem~19.7.3]{Iwaniec-Martin-Geometric-Function-Theory-Non-Linear-Analysis}), may be straightforwardly adapted to our setting, we prefer an energy based method as discussed in the introduction. This highlights the strong interactions between quasiregular curves and the isoperimetric profile of the target as utilized in the previous sections and in \cite{ikonen-2025-entire-conformal-curves,ikonen-2026-pushforward-of-currents-under-sobolev-maps,ikonen-2025-quasiregular-curves-removability-of-singularities,ikonen-pankka-2024-liouvilles-theorem-in-calibrated-geometries,pankka-pim-2025-nodal-resolution-of-quasiregular-curves-via-bubble-trees}.

To setup the statements, we introduce the following definition.
\begin{definition}\label{definition-zalcman-sequence}
Let \(\Omega \subset \R^n\) be an open set, and let \((N,\omega)\) be an \(n\)-calibrated  \(m\)-manifold.
Let $\mathcal{F}$ be a family of $K$-quasiregular curves $\Omega \to (N,\omega)$ for \(K \geq 1\). 
We say that \((G_j)_j\) is a \emph{Zalcman sequence for \(\mathcal{F}\) at \(x \in \Omega\)} if $G_j(y) \coloneqq F_j( x_j + \rho_j y )$ for sequences \((F_j)_j\) in \(\mathcal{F}\), $\rho_j \rightarrow 0^{+}$ in $(0,\infty)$, and $x_j \rightarrow x$ in $\Omega$.
\end{definition}
Observe that \(\rho_j \to 0\) and \(x_j \to x\) imply that the domains of \(G_j\) converge to \(\R^n\) as \(j \to \infty\). Additionally, we say that \((\rho_j)_j\), \((x_j)_j\), and \((F_j)_j\) are the \emph{sequences corresponding} to \((G_j)_j\).

We state a stronger version of \Cref{theorem-rescaling-principle} for non-compact targets:

\begin{theorem}\label{theorem-rescaling-principle-equiboundedness-build-in}
    Let \(M\) be an oriented \(n\)-manifold, and let \((N,\omega)\) be an \(n\)-calibrated \(m\)-manifold with an injectivity radius lower bound $2\delta$ and a sectional curvature upper bound $( \pi/(2\delta) )^2$. Let $\mathcal{F}$ be an equibounded family of $K$-quasiregular curves $M \to (N,\omega)$ for \(K \geq 1\).
    
    If $\mathcal{F}$ is not normal at $x \in M$, then there exists a non-constant bounded $K$-quasiregular curve $\R^n \to ( N, \omega )$ that is uniformly Hölder continuous. Moreover, if every element of $\mathcal{F}$ maps into a closed set $C \subset N$, so does the entire map.
\end{theorem}

This leads to one of the main applications of our rescaling principle, which is a characterization, in terms of normal families, of calibrated manifolds \((N,\omega)\) which satisfy Liouville's theorem for bounded curves: \emph{every bounded entire quasiregular curve into \((N,\omega)\) is constant}; cf. \cite[Theorem 1.7]{pankka-quasiregular-curves}.
\Cref{theorem-qr-ellipticity-equivalent-to-normality} is the special case of this characterization for compact targets.

\begin{theorem}\label{theorem-liouville-property-equivalent-normality}
    Let \((N,\omega)\) be an \(n\)-calibrated  \(m\)-manifold with an injectivity radius lower bound \(2\delta\) and a sectional curvature upper bound \((\pi/(2\delta))^2\), and $K \geq 1$. Then the following are equivalent.
    \begin{enumerate}
        \item every bounded $K$-quasiregular curve $\R^n \to (N,\omega)$ is constant;

        \item every equibounded family of $K$-quasiregular curves $M \to (N,\omega)$ is a normal family, for each oriented $n$-manifold $M$; 
        
        \item every equibounded family of $K$-quasiregular curves $\R^n \supset \mathbb{B} \to (N,\omega)$ is normal.
    \end{enumerate}
\end{theorem}

Both of these results are a consequences of the following contrapositive equicontinuity and normality criteria, which add to those found in \Cref{theorem-equicontinuous-family} and \Cref{corollary-normal-family}. They are based on the existence of Zalcman sequences with certain stronger properties.

\begin{theorem}\label{theorem-energy-rescaling-principle}
    Let $\Omega \subset \R^n$ be an open set, and let \((N,\omega)\) be an \(n\)-calibrated  \(m\)-manifold with an injectivity radius lower bound $2\delta$ and a sectional curvature upper bound $( \pi/(2\delta) )^2$. Let \(\mathcal{F}\) be a family of \(K\)-quasiregular curves \(\Omega \to (N,\omega)\) for $K \geq 1$. Then the following are equivalent:
\begin{enumerate}
    \item \(\mathcal{F}\) is not equicontinuous at \(x \in \Omega\);

    \item there exists a Zalcman sequence \((G_j)_j\) at \(x \in \Omega\) which is equicontinuous and satisfies both 
    \[
    \liminf_{j \to \infty} E(G_j|_{\mathbb{B}(0,1)}) \geq E_N \quad \text{and} \quad \limsup_{j \to \infty} E( G_j|_{\mathbb{B}(y,1/4)} ) \leq E_N 
    \]
    for every $y \in \R^n$.
\end{enumerate}
    Moreover, when \(\mathcal{F}\) is equibounded in space, the following are equivalent:
    \begin{enumerate}
        \setcounter{enumi}{2}
        \item \(\mathcal{F}\) is not normal at \(x \in \Omega\);
        \item there exists a Zalcman sequence \((G_j)_j\) at \(x \in \Omega\) which converges locally uniformly to a bounded non-constant \(K\)-quasiregular curve \(G\colon \R^n \to (N,\omega)\).
    \end{enumerate}
\end{theorem}

\begin{remark}
As in \cite{miniowitz-1982-normal-families-of-quasimeromorphic-mappings} and \cite{zalcman-1975-a-heuristic-principle-in-complex-function-theory}, our rescaling principle \Cref{theorem-energy-rescaling-principle} may be restated to immediately give a ``\emph{rigorous Bloch principle}'' for quasiregular curves. See \cite{bergweiler-2006-blochs-principle} for applications and historical discussion. See also \cite[Section $\text{6.42}_{+}$]{gromov-2007-metric-structures-for-riemannian-and-non-riemannian-spaces}.
\end{remark}

Most of the section is spent on proving \Cref{theorem-energy-rescaling-principle}. However, to emphasize its utility, we first present the proofs of the other two theorems stated above.
\begin{proof}[Proof of \Cref{theorem-rescaling-principle-equiboundedness-build-in}]
    For every $K' > K$, there exists an open neighbourhood $V$ of the origin in $T_xM$ such that $F \circ \mathrm{exp}_{x}|_{V} \colon V \to ( N, \omega )$ is a $K'$-quasiregular curve for any $K$-quasiregular curve $F \colon M \to ( N, \omega )$. This follows from the local bi-Lipschitz bounds of the exponential map. Now, the family \(\mathcal{F}' = \{ F \circ \mathrm{exp}_{x}|_{V}  \mid F \in \mathcal{F} \}\) consists of $K'$-quasiregular curves \(V \to ( N, \omega )\). For a sufficiently small $V$, the family is equibounded in space and clearly fails to be normal at the origin. A Zalcman sequence $( G_j )_j$ which converges locally uniformly to a bounded and non-constant $K'$-quasiregular curve $\R^n \to ( N, \omega )$ is obtained from \Cref{theorem-energy-rescaling-principle}. As the exponential map is asymptotically an isometry, the limit is a $K$-quasiregular curve. The conclusion on the image of the map is also clear.
\end{proof}

\begin{proof}[Proof of \Cref{theorem-liouville-property-equivalent-normality}]
     Suppose there exists an oriented $n$-manifold \(M\) and an equibounded family of $K$-quasiregular curves $M \to (N,\omega)$ which is not normal. \Cref{theorem-rescaling-principle-equiboundedness-build-in} implies that there exists a non-constant bounded $K$-quasiregular curve $\R^n \to ( N, \omega )$. Hence '$(1) \Rightarrow (2) \Rightarrow (3)$' follows.  

     Regarding '$(3) \Rightarrow (1)$', we argue as follows. Suppose there exists a bounded and non-constant \(K\)-quasiregular curve \(F \colon \R^n \to (N,\omega)\).  Consider the family \((F_j \colon \mathbb{B} \to (N,\omega))_j\) of \(K\)-quasiregular curves given by \(F_j(y) \coloneqq F(j y)\). This family is bounded since \(F\) is bounded. 
     Now set \(\rho_j = 1/j\), \(x_j = x = 0\) and \(G_j(z) = F_j(x_j + \rho_j z) = F(j \rho_j z) = F(z)\). Thus \((G_j)_j\) is a Zalcman sequence at \(0\) which converges locally uniformly to \(F\). Hence \Cref{theorem-energy-rescaling-principle} implies that \((F_j)_j\) is not normal at \(0 \in \mathbb{B}\). This completes the proof.
\end{proof}

\Cref{theorem-energy-rescaling-principle} is a corollary of our rescaling principles: \Cref{prop:energy-rescale-principle-equicont}, \Cref{prop:energy-rescale-principle-normal}, and a technical lemma. We have decomposed our rescaling principle in this manner in order both to emphasize the role of equicontinuity and to highlight that we may obtain finer control than \Cref{theorem-energy-rescaling-principle} suggests.

\begin{proposition}\label{prop:energy-rescale-principle-equicont}
    Let \(\Omega \subset \R^n\) be an open set, and let \((N,\omega)\) be an \(n\)-calibrated  \(m\)-manifold with an injectivity radius lower bound $2\delta$ and a sectional curvature upper bound $( \pi/(2\delta) )^2$. Let \(\mathcal{F}\) be a family of \(K\)-quasiregular curves \(\Omega \to (N,\omega)\) for $K \geq 1$.

    If \(\mathcal{F}\) is not equicontinuous at \(x \in \Omega\), then there exists a Zalcman sequence \((G_j)_j\) for \(\mathcal{F}\) at \(x\) and radii $( R_j )_j$ diverging to infinity which satisfy 
    \[
        E(G_j|_{\mathbb{B}}) \geq E_N \quad \text{and} \quad \sup_{ |y| \leq R_j/2 } E( G_j|_{\mathbb{B}(y,1/4)} ) \leq E_N 
    \]
    for all \(j \in \N\). In particular, $( G_j )_j$ is equicontinuous.
\end{proposition}
\begin{proof}
    Without loss of generality, we may assume \(\Omega = \mathbb{B} \subset \R^n\) and that \(\mathcal{F}\) is not equicontinuous at \(x = 0 \in \mathbb{B}\). Since \(\mathcal{F}\) is not equicontinuous at \(0\), \Cref{theorem-equicontinuous-family} implies that there exists a sequence \((\widehat{F}_j)_j \subset \mathcal{F}\) such that
    \begin{align*}
        E( \widehat{F}_j|_{ \mathbb{B}(j^{-2}) } ) > E_N \quad \text{for each}\quad j\in \N.
    \end{align*}
    Define $F_j \colon \mathbb{B} \to ( N, \omega )$ by $z \mapsto \widehat{F}_j(j^{-1}z)$, and notice that
    \begin{align}\label{eq:energy-rescale-blowup}
        E( F_j|_{ \mathbb{B}(j^{-1}) } ) > E_N \quad \text{for each}\quad j\in \N.
    \end{align}
    
    For \(x \in \mathbb{B}(1/2)\), define
    \begin{equation}\label{eq:energy-rescale-parameter-rho-and-lambda-j}
        \hat\rho_j(x) \coloneqq \inf\Set{ 1/2 \geq \rho > 0 \mid E(F_j|_{\mathbb{B}(x,\rho)}) \geq E_N }, \quad 
        \lambda_j(x) \coloneqq \frac{1/2 - \abs{x}}{\hat\rho_j(x)},
    \end{equation}
    and set \(\Lambda_j = \sup_{x \in \mathbb{B}(1/2)} \lambda_j(x)\). Let \((x_j)_j\) be a sequence in \(\mathbb{B}(1/2)\) which almost achieves \(\Lambda_j\), i.e. \(\lambda_j(x_j) \geq \Lambda_j/2\). Finally, set \(\rho_j = \hat\rho_j(x_j)\) and \(R_j = \lambda_j(x_j)\).
    
    Observe that \eqref{eq:energy-rescale-blowup} implies \(\hat\rho_j(0) \to 0\), and thus
    \[
    \Lambda_j = \sup_{x \in \mathbb{B}(1/2)} \lambda_j(x) \geq \lambda_j(0) = \frac{1/2}{\hat\rho_j(0)} \to \infty
    \]
    as \(j \to \infty\). Hence \(R_j \to \infty\) and \(\rho_j \to 0\). Define \(G_j(y) \coloneqq F_j(x_j + \rho_j y)\) for \(y \in \mathbb{B}(R_j)\). Since \(R_j \to \infty\), each \(y \in \R^n\) satisfies \(\abs{y} \leq R_j\) for all \(j\) large enough. Clearly each mapping \(G_j \colon \mathbb{B}(R_j) \to (N,\omega)\) is a \(K\)-quasiregular curve, and the construction gives that $( G_j )_j$ is a Zalcman sequence for $\mathcal{F}$ at $0$ with parameters $j^{-1}x_j \rightarrow 0$ and $j^{-1}\rho_j \rightarrow 0$.

    For the energy bounds for \((G_j)_j\), let $j \in \N$ and $y \in \R^n$ with $\abs{y} \leq R_j/2$.
    Observe that for \(z \in \mathbb{B}(\frac{1}{2})\),
    \[
        \hat\rho_j(z) = \frac{1/2 - \abs{z}}{\lambda_j(z)} \geq \frac{1/2 - \abs{z}}{2 \lambda_j(x_j)} = \frac{\rho_j}{2} \frac{1/2 - \abs{z}}{1/2 - \abs{x_j}},
    \]
    and that \(x_j + \rho_jy \in \mathbb{B}(\frac{1}{2})\) by the triangle inequality since
    \[
    \abs{\rho_jy} \leq  \frac{\rho_j R_j}{2} = \frac{1}{4} -  \frac{\abs{x_j}}{2}.
    \]
    Substituting \(z = x_j + \rho_j y\), we estimate that
    \begin{align*}
        \hat\rho_j(x_j + \rho_j y) &\geq \frac{\rho_j}{2} \frac{1/2 - \abs{x_j + \rho_j y}}{1/2 - \abs{x_j}} \geq \frac{\rho_j}{2} \left(\frac{1/2 - \abs{x_j}}{1/2 - \abs{x_j}} - \frac{\rho_j\abs{y}}{1/2 - \abs{x_j}}\right) \\
        &= \rho_j \left(\frac{1}{2} - \frac{\abs{y}}{2 R_j} \right)  \geq \frac{\rho_j}{4},
    \end{align*} 
    where we used that \(1/R_j = \rho_j/(1/2 - \abs{x_j})\) and \(\abs{y} \leq R_j/2\).
    Thus, by \eqref{eq:energy-rescale-parameter-rho-and-lambda-j}, 
    \begin{equation}\label{eq:energy-rescale-G-j-small-energy}
        E(G_j|_{\mathbb{B}(y,\frac{1}{4})}) = E\left (F_j|_{\mathbb{B}\left(z,  \frac{\rho_j}{4}\right)} \right) \leq E( F_j|_{ \mathbb{B}( z, \hat{\rho}_j(z) )} ) \leq E_N.
    \end{equation}
    As \(y\) was an arbitrary point with $\abs{y} \leq R_j/2$ and $\lim_{j} R_j = \infty$, \Cref{theorem-equicontinuous-family} shows that \((G_j)_j\) is equicontinuous. Additionally, \eqref{eq:energy-rescale-parameter-rho-and-lambda-j} implies that
    \[
    E(G_j|_{\mathbb{B}}) = E(F_j|_{\mathbb{B}(x_j,\rho_j)}) \geq E_N
    \]
    for all \(j\). The claim follows.
\end{proof}

Following \Cref{prop:energy-rescale-principle-equicont}, we record a number of related contrapositive statements in the following technical lemma. 

\begin{lemma}\label{lem:zalcman-seq-at-normal-point}
    Let $\Omega \subset \R^n$ be an open subset, and let \((N,\omega)\) be an \(n\)-calibrated  \(m\)-manifold with an injectivity radius lower bound $2\delta$ and a sectional curvature upper bound $( \pi/(2\delta) )^2$. Let \(\mathcal{F}\) be a family of \(K\)-quasiregular curves \(\Omega \to (N,\omega)\) for $K \geq 1$.
    \begin{enumerate}
        \item If \(\mathcal{F}\) is equicontinuous at \(x \in \Omega\), then every Zalcman sequence \((G_j)_j\) at \(x\) is equicontinuous and satisfies
        \[
        \limsup_{j\to \infty} E(G_j|_{\mathbb{B}(y,r)}) = 0 \quad \text{for all}\quad r> 0 \quad \text{and} \quad y \in \R^n.
        \] 

        \item If \(\mathcal{F}\) is equibounded in space, then every Zalcman sequence \((G_j)_j\) at any \(x \in \Omega\) is also equibounded in space. In particular, there exists a compact set \(D \subset N\) satisfying the following:  for \(s > 0\), there exists \(j_s \in \N\) such that
        \[
        G_j(\overline{\mathbb{B}}(s)) \subset D \quad \text{for all} \quad j \geq j_s.
        \]

        \item If \(\mathcal{F}\) is normal at \(x \in \Omega\), then every Zalcman sequence \((G_j)_j\) at \(x\) is normal and each subsequential local uniform limit of \((G_j)_j\) is constant.
    \end{enumerate}
\end{lemma}
\begin{proof}
    Without loss of generality, we may assume that \(\Omega = \mathbb{B} \subset \R^n\) and that \((G_j)_j\) is a Zalcman sequence for \(\mathcal{F}\) at \(0\) with parameters \((\rho_j)_j\), points \((x_j)_j\), and \((F_j)_j\) being the corresponding sequences.

    \emph{For (1):} Assume that \(\mathcal{F}\) is equicontinuous at $0$. Let \(0 < \varepsilon < E_N\). Since \(\mathcal{F}\) is equicontinuous at \(0\), by \Cref{theorem-equicontinuous-family}, there exists \(R > 0\) for which $$\sup_{F \in \mathcal{F}} E(F|_{\mathbb{B}(R)}) < \varepsilon.$$ Take $y \in \R^n$ and $r > 0$. As \(x_j \to 0\) and \(\rho_j \to 0\), we have that \(\mathbb{B}(x_j + \rho_j y, r \rho_j) \subset \mathbb{B}(R)\) for all \(j\) large enough. Hence
    \begin{align*}
        \limsup_{j\to\infty} E(G_j|_{\mathbb{B}(y, r)}) &\leq \limsup_{j\to\infty} E(F_j|_{\mathbb{B}(R)}) \leq \sup_{F \in \mathcal{F}} E(F|_{\mathbb{B}(R)}) < \varepsilon.
    \end{align*}
    As \(\varepsilon > 0\) was arbitrary, the claim follows.

    \emph{For (2):} Assuming \(\mathcal{F}\) is equibounded in space, for each $R \in (0,1)$, there exists a compact subset \(D \subset N\) such that 
    \(\bigcup_{F \in \mathcal{F}} F(\overline{\mathbb{B}}(R)) \subset D\).
    Consider \(s > 0\). Since \(\rho_j \to 0\) and \(x_j \to 0\), there exists \(j_0 \in \N\) for which \(\overline{\mathbb{B}}(x_j, s\rho_j) \subset \mathbb{B}(R)\) and 
    \[
    G_j(\overline{\mathbb{B}}(s)) = F_j(\overline {\mathbb{B}}(x_j, s\rho_j)) \subset F_j(\overline{\mathbb{B}}(R)) \subset D
    \]
    for all \(j \geq j_0\). Hence \((G_j)_j\) is equibounded in space and the claim follows.

    \emph{For (3):} Assume that \(\mathcal{F}\) is normal at $0$. Now \Cref{lem:zalcman-seq-at-normal-point} (1)--(2) above imply that \((G_j)_j\) is normal. Let \((G_{j_\ell})_\ell\) be a subsequence of \((G_j)_j\) which converges locally uniformly to a \(K\)-quasiregular curve \(G\colon \R^n \to (N,\omega)\). By \Cref{lem:zalcman-seq-at-normal-point} (1) and \Cref{corollary-convergence-of-pullbacks}, 
    \[
    E(G|_{\mathbb{B}(r)}) \leq \liminf_{\ell\to\infty} E(G_{j_\ell}|_{\mathbb{B}(r)}) \leq \limsup_{\ell\to\infty} E(G_{j_\ell}|_{\mathbb{B}(r)}) = 0
    \]
    for \(r > 0\). Letting \(r \to \infty\), we see that \(E(G) = 0\). Hence \(G\) is constant and the claim follows.
\end{proof}
The following proposition is the last part needed for the proof of \Cref{theorem-energy-rescaling-principle}.

\begin{proposition}\label{prop:energy-rescale-principle-normal}
    Let \(\Omega \subset \R^n\) be an open set, and let \((N,\omega)\) be an \(n\)-calibrated  \(m\)-manifold with an injectivity radius lower bound $2\delta$, a sectional curvature upper bound $( \pi/(2\delta) )^2$, and H\"older data \(E_N\) and \(\beta\) for $K \geq 1$. Let \(\mathcal{F}\) be an equibounded family of \(K\)-quasiregular curves \(\Omega \to (N,\omega)\).

    If \(\mathcal{F}\) is not normal at \(x \in \Omega\), then there exists a Zalcman sequence \((G_j)_j\) at \(x\) which converges locally uniformly to a bounded non-constant \(\beta\)-Hölder continuous \(K\)-quasiregular curve \(G \colon \R^n \to (N,\omega)\).
\end{proposition}

\begin{proof}
    Without loss of generality, we may assume that \(\Omega = \mathbb{B}\) and \(x = 0\).
    Since \(\mathcal{F}\) is equibounded in space but not normal at \(0\), the Arzelà--Ascoli theorem implies it is not equicontinuous at \(0\). Then \Cref{prop:energy-rescale-principle-equicont} gives the existence of an equicontinuous Zalcman sequence \((G_j)_j\) at \(0\) and radii $(R_j)_j$ diverging to $\infty$ such that
    \[
        E(G_j|_{\mathbb{B}}) \geq E_N \quad \text{and} \quad \sup_{ |y| \leq R_j/2} E( G_j|_{\mathbb{B}(y,1/4)} ) \leq E_N \quad \text{for all \(j\).}
    \] 
    
    Let \((\rho_j)_j\), \((x_j)_j\) and \((F_j)_j\) be the sequences corresponding to \((G_j)_j\).
    By \Cref{lem:zalcman-seq-at-normal-point} (2) and the Arzelà--Ascoli theorem, \((G_j)_j\) is normal.
    By passing to a convergent subsequence and relabeling, \((G_j)_j\) converges locally uniformly to a \(K\)-quasiregular curve \(G\colon \R^n \to (N,\omega)\). \Cref{corollary-convergence-of-pullbacks} additionally implies that \(\star G^*_j \omega \to \star G^*\omega\) in \(L^{p/n}(U)\) for the exponent \(p > n\) from that corollary, and for each precompact domain \(U \subset \R^n\). Thus the energy lower bound for \(G_j\) implies
    \[
    0 < \frac{c}{K} E_N \leq \frac{c}{K}\limsup_{j\to\infty} E(G_j|_{\mathbb{B}(0,1)}) \leq \lim_{j\to\infty} \int_{\mathbb{B}(0,1)} G_j^*\omega = \int_{\mathbb{B}(0,1)} G^* \omega.
    \]
    Hence \(G\) is non-constant. Note that \Cref{lem:zalcman-seq-at-normal-point} (2) additionally gives the existence of a compact set \(D \subset N\) satisfying that for each \(r > 0\), \(G_j(\overline{\mathbb{B}}(0,r)) \subset D\) for all \(j\) large enough. In particular, this implies \(G(\R^n) \subset D\) and hence that \(G\) is bounded.

    We prove that $G$ is $\beta$-Hölder. Observe that for each \(z \in \R^n\), \Cref{theorem-modulus-of-continuity-estimates} and the energy upper bound for \((G_j)_j\) imply
    \[
    d(G_j(x), G_j(y)) \leq C(n,\beta) E_N^{1/n} 4^{\beta} \abs{x - y}^\beta
    \]
    for all \(x,y \in \mathbb{B}(z,1/8)\) and sufficiently large $j$. Since \(G_j\) converges locally uniformly to \(G\), we see that
    \[
    d(G(x), G(y)) \leq C(n,\beta) E_N^{1/n} 4^{\beta} \abs{x - y}^\beta
    \]
    for all \(x,y \in \R^n\) with \(\abs{x - y} < 1/4\). When \(\abs{x - y} \geq 1/4\), we estimate 
    \[
    d(G(x), G(y)) \leq (\diam D) \leq (\diam D) 4^\beta \abs{x - y}^\beta.
    \]
    Combining these two estimates completes the proof.
\end{proof}

\begin{proof}[Proof of \Cref{theorem-energy-rescaling-principle}]
    \Cref{prop:energy-rescale-principle-equicont} and \Cref{lem:zalcman-seq-at-normal-point} give a proof of `(1) \(\Leftrightarrow\) (2)' of \Cref{theorem-energy-rescaling-principle} and \Cref{prop:energy-rescale-principle-normal} completes the proof of `(3) \(\Leftrightarrow\) (4)' of the theorem. In fact, in `(3) $\Rightarrow$ (4)', the limiting map is not only non-constant but bounded and $\beta$-Hölder.
\end{proof}

In addition to \Cref{theorem-rescaling-principle-equiboundedness-build-in,theorem-liouville-property-equivalent-normality}, we have the following applications of our rescaling principle, \Cref{prop:energy-rescale-principle-normal}.
To begin with, combining \Cref{theorem-liouville-property-equivalent-normality} and \Cref{prop:energy-rescale-principle-normal} shows that the existence of a bounded entire quasiregular curve implies the existence of a Hölder continuous one; the corollary was anticipated by Gromov in \cite[Section $\text{6.42}_{+}$]{gromov-2007-metric-structures-for-riemannian-and-non-riemannian-spaces} for quasiminimal quasiconformal maps which are closely related to quasiregular curves. See also \cite[Section 2]{bonk-heinonen-2001-quasiregular-mappings-and-cohomology} for discussion on quasiregular mappings in this context.

\begin{corollary}\label{corollary-rescale-entire-qrc-to-uniformly-holder}
    Let \((N,\omega)\) be an \(n\)-calibrated  \(m\)-manifold with an injectivity radius lower bound $2\delta$, a sectional curvature upper bound $( \pi/(2\delta) )^2$, and H\"older data \(E_N\) and \(\beta\) for $K \geq 1$.   
    
    Suppose there exists a bounded non-constant \(K\)-quasiregular curve \(f \colon \R^n \to (N,\omega)\). Then there exists a bounded, non-constant, and \(\beta\)-Hölder continuous \(K\)-quasiregular curve \(F \colon \R^n \to (N,\omega)\).
\end{corollary}

We also obtain an easy proof of \cite[Corollary~1.3]{pankka-pim-2025-nodal-resolution-of-quasiregular-curves-via-bubble-trees}.
\begin{corollary}\label{corollary-energy-rescaling-principle-aspherical-not-normal-implies-infinite-energy}
    Let \(\Omega \subset \R^n\) be an open set, and let \((N,\omega)\) be an \(n\)-calibrated  \(m\)-manifold with an injectivity radius lower bound $2\delta$ and a sectional curvature upper bound $( \pi/(2\delta) )^2$. Suppose that $\omega$ is aspherical, and let \(\mathcal{F}\) be a family of \(K\)-quasiregular curves \(\Omega \to (N,\omega)\) for $K \geq 1$.
    
    If \(\mathcal{F}\) is equibounded but not normal at \(x \in \Omega\), then 
    \[
    \sup_{F \in \mathcal{F}} E(F|_{\mathbb{B}(x,s)}) = \infty \quad \text{for all sufficiently small $s > 0$.}
    \]
\end{corollary}
\begin{proof}
Without loss of generality, we may assume \(\Omega = \mathbb{B}\) and \(x = 0 \in \mathbb{B}\).  

Take \(0 < s < 1\) small enough that the restricted family \(\mathcal{F}|_{\mathbb{B}(s)} = \set{F|_{\mathbb{B}(s)} \mid F \in \mathcal{F}}\) is equibounded. Then by \Cref{theorem-energy-rescaling-principle}, there exists a Zalcman sequence \((G_j)_j\) for \(\mathcal{F}|_{\mathbb{B}(s)}\) at 0 which converges locally uniformly to a non-constant \(K\)-quasiregular curve \(G\colon \R^n \to (N,\omega)\). 
Since \(\omega\) is aspherical, \(E(G) = \infty\) by \Cref{remark-homological-non-triviality-asherical}.

Let \((\rho_j)_j\), \((x_j)_j\), and \((F_j)_j\) be the sequences corresponding to \((G_j)_j\). Now, observe that for \(r > 0\) and \(j\) large enough, \(\mathbb{B}(x_j,r \rho_j) \subset \mathbb{B}(s)\), and hence
\[
E(G_j|_{\mathbb{B} (r)}) = E(F_j|_{\mathbb{ B }(x_j, r\rho_j)}) \leq E(F_j|_{\mathbb{B}(s)})
\]
for all \(j\) large enough.
As \(G_j \to G\) locally uniformly, we compute that
\[
    E(G|_{ \mathbb{B}(r)}) \leq \liminf_{j \to \infty} E(G_j|_{ \mathbb{B}(r)}) \leq \limsup_{j\to \infty} E(F_j|_{ \mathbb{B}(s)}) \leq \sup_{F \in \mathcal{F}} E(F|_{\mathbb{B}(s)}).
\]
Letting \(r \to \infty\), we obtain \(\infty = E(G) \leq \sup_{F \in \mathcal{F}} E(F|_{\mathbb{B}(s)})\). As \(s > 0\) was arbitrary, the claim follows.
\end{proof}

\section{Conformal curves and hyperbolicity}\label{section-smith-maps-and-hyperbolicity}
    We start the proof of \Cref{theorem-characterization-of-hyperbolicity-calibration} with the following lemma showing that conformal curves are \(n\)-harmonic on small-energy domains; see \cite[Theorem 5.1]{heikkila-pankka-prywes-2023-quasiregular-curves-of-small-distortion-in-product-manifolds}, \cite[Section 1.1]{pankka-quasiregular-curves}, \cite[Section 3.5]{cheng-karigiannis-madnick-2020-bubble-three-convergence-of-conformally-cross-product-preserving-maps}, and \cite[Theorem 6.9]{ikonen-2026-pushforward-of-currents-under-sobolev-maps} for related results.
    \begin{lemma}\label{lemma-smith-maps-are-n-harmonic}
    Let \(M\) be an oriented \(n\)-manifold, let $(N, \omega)$ be an $n$-calibrated $m$-manifold with an injectivity radius lower bound $2\delta$ and a sectional curvature upper bound $( \pi/(2\delta) )^2$, and let $F \colon M \to ( N, \omega )$ be a conformal curve. Then $F$ is locally $n$-harmonic, i.e. if $U \subset M$ is a precompact smooth domain with $\partial U \neq \emptyset$ and $E( F|_{U} ) \leq n^{-n/2}E_N$ for the constant in \eqref{equation-small-mass-isoperimetric-inequalities}, then for any $G \in W^{1,n}( U, N )$ coinciding with $F$ outside a compact subset of $U$, it holds that
    \begin{align*}
        \mathcal{E}( F|_{U} ) \leq \mathcal{E}( G ).
    \end{align*}
    \end{lemma}
    Here $\mathcal{E} = \mathcal{E}_n$ is the energy defined using the Hilbert--Schmidt norm. The above implies that $F|_{U}$ is weakly $n$-harmonic: it is a weak solution to the $n$-harmonic system of equations.
    \begin{proof}
    Since $F$ is a conformal curve, it follows that $\|DF\|^{n} = n^{-n/2} \|DF\|^{n}_{HS} = \star F^{*}\omega$ almost everywhere; see e.g. \cite[Lemma 2.1]{ikonen-pankka-2024-liouvilles-theorem-in-calibrated-geometries}. It suffices to prove that if $G \in W^{1,n}( U, N )$ agrees with $F$ outside a compact set and
    \begin{align*}
        \mathcal{E}(G)
        =
        \int_{ U } \|DG\|^{n}_{HS} \,dz
        \leq
        \int_{ U } \|DF\|^{n}_{HS} \,dz
        =
        \mathcal{E}( F|_{U} )
        =
        n^{n/2} E( F|_{U} ),
    \end{align*}
    then equality holds. We prove this using the pushforward techniques from \cite{ikonen-2026-pushforward-of-currents-under-sobolev-maps}.
    
    Let $C \subset U$ be a compact set so that $F = G$ in $U \setminus C$. As explained in the preliminaries, we may approximate $U$ from the inside by smooth domains $( U_j )_j$ containing $C$ for which the pushforwards $F_{*}[U_j]$ and $G_{*}[U_j]$ are well-defined and the boundaries of the currents agree. Now, by the mass-energy inequality of the pushforward, it holds that
    \begin{align*}
        M( G_{*}[U_j] ) 
        &\leq E( G|_{U_j} ) 
        \leq \int_{ U_j } \|DG\|^{n}_{HS} \,dz 
        \leq \int_{ U } \|DG\|^{n}_{HS} \,dz
        \\
        &\leq \int_{ U  } \|DF\|^{n}_{HS} \,dz
        = n^{n/2} E( F|_{U} )
        \leq
        E_N.
    \end{align*}
    Similarly,
    \begin{align*}
        M( F_{*}[U_j] ) \leq E( F|_{U_j} ) \leq n^{-n/2}E_N \leq E_N. 
    \end{align*}
    It follows that $T = F_{*}[U_j] - G_{*}[U_j]$ is an integral $n$-cycle with mass at most $2E_N \leq M'$. Thus, by the isoperimetric inequality, there exists an integral $(n+1)$-current whose boundary is $T$. Then, by the compatibility of the pushforward and pullback and $\omega$ being closed, Stokes' theorem implies
    \begin{align*}
        \int_{ U_j } F^{*}\omega = \int_{ U_j } G^{*}\omega.
    \end{align*}
    We conclude, therefore, that
    \begin{align*}
        n^{-n/2}
        \int_{ U_j } \|DF\|^{n}_{HS} \,dz
        =
        \int_{ U_j } F^{*}\omega
        =
        \int_{ U_j } G^{*}\omega
        \leq
        n^{-n/2}
        \int_{U_j} \|DG\|^{n}_{HS} \,dz
    \end{align*}
    by Hadamard's inequality. By passing to the limit $j \to \infty$, dominated convergence establishes that
    \begin{align*}
        \int_{ U } \|DF\|^{n}_{HS} \,dz
        \leq
        \int_{ U } \|DG\|^{n}_{HS} \,dz.
    \end{align*}
    The claimed equality holds and thus the conclusion follows.
    \end{proof}
    By \Cref{lemma-smith-maps-are-n-harmonic} and a special case of a theorem due to Hardt--Lin \cite[Theorem 3.1]{hardt-lin-1987-mappings-minimizing-the-lp-norm-of-the-gradient}, it follows that conformal curves $\mathbb{B} \to ( N, \omega )$ have locally Hölder continuous differentials. On the level of families, this can be quantified by the following version of Marty's theorem.
    \begin{proposition}\label{proposition-equicontinuity-results}
        Let $(N, \omega)$ be a closed $n$-calibrated $m$-manifold with an injectivity radius lower bound $2\delta$ and a sectional curvature upper bound $( \pi/(2\delta) )^2$. Then the following are equivalent for the family $\mathcal{F}_1(\omega)$ of all conformal curves $\mathbb{B} \to (N,\omega)$.
        \begin{enumerate}
            \item the family $\mathcal{F}_1(\omega)$ is normal;
            \item there exists a radius $s \in (0,1)$ and a constant $D > 0$ such that
            \begin{align*}
                \sup_{ F \in \mathcal{F}_{1}(\omega) } \sup_{ |z| \leq s } \|DF\|(z) \leq D.
            \end{align*}
        \end{enumerate}
    \end{proposition}
    \begin{proof}
        We simplify notation by setting $\mathcal{F} = \mathcal{F}_1(\omega)$. In the subsequent argument, we use (orientation-preserving) Möbius automorphisms of the unit ball $\mathbb{B}$. For the start of the proof, we recall some basic facts from hyperbolic geometry.
        
        We recall that Möbius transformations send hyperbolic balls to hyperbolic balls (when equipping $\mathbb{B}$ with the standard hyperbolic metric). Furthermore, for each $z \in \mathbb{B}$, there exists a suitably normalized Möbius automorphism $T_z$ of $\mathbb{B}$ sending $z$ to $0$, cf. \cite[Chapter I, Section 2, pages 10--11]{vuorinen-1988-conformal-geometry-and-quasiregular-mappings}. These maps have the property that if $|z| \leq r$ for $r \in (0,1)$, then $T_z$ is $L$-bi-Lipschitz for $L = L(r)$ \cite[Chapter I, Section 1.39, pages 11--12]{vuorinen-1988-conformal-geometry-and-quasiregular-mappings}.
        
        For the comparison between hyperbolic balls $\mathbb{B}_{\mathrm{hyp}}(y_0,\rho)$ and Euclidean balls $\mathbb{B}(x_0,r)$ we need below, we refer the reader to \cite[Chapter I, Section 2, pages 24--27]{vuorinen-1988-conformal-geometry-and-quasiregular-mappings} for details. 
        We recall that each hyperbolic ball is Euclidean, in the sense that
        \begin{equation}\label{equation-equality-of-hyperbolic-and-euclidean-balls}
            \mathbb{B}_{\mathrm{hyp}}( y_0, \rho ) = \mathbb{B}( x_0, r )
        \end{equation}
        for some point $x_0 \in \mathbb{B}$ and radius $r >0$, and each Euclidean ball precompact in $\mathbb{B}$ is a hyperbolic ball. Moreover,
        \begin{align*}
            \mathbb{B}( y_0, a(1-|y_0|) ) \subset \mathbb{B}_{\mathrm{hyp}}( y_0, \rho ) \subset \mathbb{B}( y_0, A(1-|y_0| ) )
        \end{align*}
        where
        \begin{align*}
            a = \frac{ t (1+|y_0|) }{ 1+|y_0|t},
            \quad
            A = \frac{ t (1+|y_0|) }{ 1-|y_0|t },
            \quad\text{and}\quad
            t = \tanh\left( \frac{1}{2}\rho \right)
        \end{align*}
        for the hyperbolic tangent $\tanh$.

        Using these facts, observe that for each $\mathbb{B}(x_0,2R) \subset \mathbb{B}$, there exists a constant $c = c( |x_0|, R ) > 1$ such that if $r \in (0,c^{-1}R)$, then 
        \begin{align*}
            \mathbb{B}_{\mathrm{hyp}}( x_0, c^{-1} \rho ) \subset \mathbb{B}( x_0, r ) \subset \mathbb{B}_{\mathrm{hyp}}( x_0, \rho )
        \end{align*}
        for some radius $\rho = \rho(r) > 0$. Now $T_{x_0}( \mathbb{B}_{\mathrm{hyp}}( x_0, c^{-1} \rho ) )$ is a hyperbolic ball centered at $0$ and of radius $c^{-1}\rho$; similarly when the radius is replaced by $\rho$. Since $r \rightarrow 0^{+}$ implies $\rho \rightarrow 0^{+}$ above, we conclude that if there exists $x_0 \in \mathbb{B}$ and $\mathbb{B}( x_0, r ) \subset \mathbb{B}$ with
        \begin{align}\label{equation-assumption-on-energy-non-origin-centered}
            \sup_{ F \in \mathcal{F} } E( F|_{ \mathbb{B}( x_0, r ) } ) \leq n^{-n/2}E_N,
        \end{align}
        then
        \begin{align}\label{equation-assumption-on-energy}
            \sup_{ F \in \mathcal{F} } E( F|_{ \mathbb{B}(0,\rho) } ) \leq n^{-n/2}E_N
            \quad\text{for some $\rho \in (0,1)$.}
        \end{align}
        Conversely, if \eqref{equation-assumption-on-energy} holds, then for every $x_0 \in \mathbb{B}$, there exists $\mathbb{B}( x_0, r ) \subset \mathbb{B}$ for which \eqref{equation-assumption-on-energy-non-origin-centered} holds. This equivalence uses the additional fact that $\mathcal{F}$ is closed under precomposition by the Möbius automorphisms of $\mathbb{B}$ and the precomposition preserves energy. By \Cref{corollary-normal-family}, we see that the normality of $\mathcal{F}$ and the statement in \eqref{equation-assumption-on-energy} are equivalent.

        We start the proof of the equivalence, so suppose that (1) holds. Now, by \eqref{equation-assumption-on-energy} and \Cref{theorem-modulus-of-continuity-estimates}, there exist constants $A > 0$ and $\beta \in (0,1)$ such that 
        \begin{align*}
            d( F(x), F(y) ) \leq A ( n^{-n/2}E_N )^{1/n} \rho^{-\beta} |x-y|^{\beta} \quad\text{if $|x|,|y| < \rho/2$}
        \end{align*}
        for every $F \in \mathcal{F}$. Hence the family $\mathcal{F}' = \{ F \colon \mathbb{B}(\rho/2) \to ( N, \omega ) \mid F \in \mathcal{F} \}$ is uniformly Hölder and $n$-harmonic by \Cref{lemma-smith-maps-are-n-harmonic}. Since \(N\) is compact, by a special case of \cite[Theorem 3.1]{hardt-lin-1987-mappings-minimizing-the-lp-norm-of-the-gradient}, there exists an exponent $\alpha \in (0,1)$ and a constant $C > 0$ such that 
        \begin{align*}
            \left| \|DF\|(x) - \|DF\|(y) \right| \leq C |x-y|^{\alpha} \quad\text{if $|x|,|y| < \rho/4$.}
        \end{align*}
        Let $D = \sup_{ F \in \mathcal{F} } \sup_{|x|\leq \rho/4}\|DF\|(x)$.
        
        We claim that $D < \infty$. Indeed, the Möbius automorphisms $T_z$ for $z \in \mathbb{B}(\rho/4)$ are $L$-bi-Lipschitz for some $L = L(\rho)$ so $D$ is comparable to
        \begin{align*}
            D_{0} = \sup_{ F \in \mathcal{F} } \|DF\|(0).
        \end{align*}
        To see that $D_0$ is bounded, notice that by the choice of $\alpha$ and $C$, it holds that
        \begin{align*}
            \left| \left( E( F|_{ \mathbb{B}(\rho/4) } ) \right)^{ \frac{1}{n} } - \left( \omega_n\left( \frac{\rho}{4} \right)^n \|DF\|^{n}(0) \right)^{ \frac{1}{n} } \right|
            \leq
            C \omega_n^{ \frac{1}{n} } \left( \frac{\rho}{4} \right)^{1+\alpha}.
        \end{align*}
        To finish, it suffices to note that $\sup_{ F \in \mathcal{F} } E( F|_{ \mathbb{B}(\rho/4) } )$ is finite by \eqref{equation-assumption-on-energy}.

        Finally, if the bound on the differential holds, clearly \eqref{equation-assumption-on-energy} holds for a sufficiently small $\rho$ as well. By the above, this is equivalent to the normality of $\mathcal{F}$.
    \end{proof}

    \begin{proof}[Proof of \Cref{theorem-characterization-of-hyperbolicity-calibration}]
    Our strategy is the following: We prove that $\omega$-hyperbolicity implies $R_\omega$-hyperbolicity. We also show that if $\omega$-hyperbolicity does not hold, then neither does $R_\omega$-hyperbolicity nor $K_{\omega}$-hyperbolicity.  Since $R_\omega$-hyperbolicity clearly implies $K_\omega$-hyperbolicity, this demonstrates that all three notions are equivalent.

    We first prove that $\omega$-hyperbolicity implies $R_\omega$-hyperbolicity via contraposition. To this end, suppose that \eqref{equation:royden-strong} fails. Then, since $\mathcal{K}_{(N,\omega)}$ is positively homogeneous \cite[Proposition 5.2]{broder-iliashenko-madnick-2025-hyperbolicity-and-schwarz-lemmas-in-calibrated-geometry}, there exist unit vectors $v_j \in T_{ x_j }N$ for $x_j \in N$ such that
    \begin{align*}
        \lim_{ j \to \infty } \mathcal{K}_{ (N,\omega) }( v_j ) = 0.
    \end{align*}
    By definition, we find $( F_j )_j$ in $\mathcal{F}_1(\omega)$ with $( DF_j )_0(w_j) = a_{j}^{-1}v_j$ for some unit vectors $w_j$ and constants $a_j \rightarrow 0^{+}$. Since $F_j$ is a conformal curve, it holds that
    \begin{align*}
        a_j^{-n} = \|DF_j\|^{n}(0).
    \end{align*}
    As $a_j \rightarrow 0^{+}$, it follows that $\mathcal{F}_1( \omega )$ cannot be normal by \Cref{proposition-equicontinuity-results}. Thus $\omega$-hyperbolicity fails by \Cref{theorem-energy-rescaling-principle}, where we recall that conformal curves are also closed under local uniform convergence.

    Next, suppose that $(N,\omega)$ is not $\omega$-hyperbolic, i.e. there exists a non-constant conformal curve $F \colon \R^n \to ( N,\omega)$. Hence there exists a point, say, the origin, where the differential of $F$ is non-degenerate. Then, letting $v = (DF)_0(e_1)$ and $F_r = F( \cdot / r )$, the chain rule implies $(DF_r)_0(e_1) = v/r$. Letting $r \rightarrow 0^{+}$, by definition, we have $\mathcal{K}_{(N,\omega)}(v) = 0$. In fact, clearly $\mathcal{K}_{(N,\omega)}$ is degenerate on the image of $(DF)_0$ (which is an $n$-dimensional subspace of $T_{F(0)}N$). Now, consider the open set $U = \{ x \in \R^n \colon (DF)_x \neq 0 \}$ and its image $D = F( U )$, and let $\gamma \colon [a,b] \to U$ be a piecewise $\mathcal{C}^{1}$-regular curve with $x \coloneqq ( F \circ \gamma )(a) \neq ( F \circ \gamma )(b) \eqqcolon y$. The argument above implies that $\mathcal{K}_{(N,\omega)}( ( F \circ \gamma )'(t) ) = 0$ for those  $t \in [a,b]$ at which $\gamma$ is smooth so $d_{ (N,\omega) }( x, y ) = 0$. Hence $K_\omega$-hyperbolicity fails. Obviously the $R_\omega$-hyperbolicity also fails.
    \end{proof}

\section{Examples of quasiregular ellipticity}\label{section-examples}
    In this section, we describe examples of quasiregularly elliptic \(n\)-calibrated \(m\)-manifolds $( N, \omega )$; here \(2 \leq n \leq m\).
    We mainly focus on the lower-dimensional cases: $2 \leq n \leq 4$ and $6 \leq m \leq 8$. To generate the examples below, we use the following elementary lemma.
\begin{lemma}\label{lemma-isometric-inclusions-and-smith-maps}
    Let $(N, \omega)$ be an $n$-calibrated $m$-manifold and let $M$ be an oriented $n$-manifold. If $\Sigma \subset (N,\omega)$ is \(\omega\)-calibrated and $f \colon M \to \Sigma$ is a $K$-quasiregular mapping, then \(\iota \circ f \colon M \to (N,\omega)\) is a \(K\)-quasiregular curve where $\iota\colon \Sigma \xhookrightarrow{} (N,\omega)$ is the inclusion. 
    In particular, if \((\Sigma,\vol_{\Sigma})\) is $K$-quasiregularly elliptic, then so is \((N,\omega)\).
\end{lemma}
    We recall that an $n$-dimensional submanifold $\Sigma \subset (N,\omega)$ is \emph{$\omega$-calibrated} if, for each $x \in \Sigma$, the tangent space $T_x\Sigma \subset T_xN$ has an orthonormal basis $v_1,\dots,v_n$ for which $\omega( v_1 \wedge \dots \wedge v_n ) =1$. Equivalently, the inclusion $\Sigma \xhookrightarrow{} (N,\omega)$ is a conformal curve. 

The following theorem by Alexander from piecewise linear geometry gives important examples of quasiregular mappings \cite{alexander-1920-note-on-riemann-spaces}.
\begin{theorem}\label{theorem-alexander-branched-cover}
    Let $\Sigma$ be a closed and oriented (Riemannian) $n$-manifold. Then there exists a surjective quasiregular mapping $\Sigma \to \S^n$ that has bounded length distortion.
\end{theorem}
    Note that \Cref{theorem-alexander-branched-cover} is not the original formulation due to Alexander but a direct consequence; we refer to the introduction of \cite{heinonen-rickman-2002-geometric-branched-covers-between-generalized-manifolds} for an in-depth discussion. We mainly apply \Cref{theorem-alexander-branched-cover} for the torus $\Sigma = \mathbb{T}^n$, in which case there is an elementary construction.
    
    We recall that a map $f \colon M \to N$ between manifolds has \emph{bounded length distortion} if there exists a constant $L > 0$ such that for every curve $\gamma \colon [a,b] \to M$, the length $\ell( f \circ \gamma )$ satisfies $L^{-1} \ell( \gamma ) \leq \ell( f \circ \gamma ) \leq L \ell( \gamma )$. The reason we highlight this property is that the product map of quasiregular mappings that have bounded length distortion is also a quasiregular map with bounded length distortion. Thus the subsequent examples and discussion easily extend to product-type examples.

    To apply \Cref{theorem-alexander-branched-cover} and \Cref{lemma-isometric-inclusions-and-smith-maps}, the next task is to find examples of calibrated submanifolds. In particular, we focus on calibrated manifolds with special holonomy groups from Berger's list such as Calabi--Yau manifolds, $G_2$ manifolds, and Kähler manifolds; cf. \cite[Chapter 3]{joyce-2007-riemannian-holonomy-groups-and-calibrated-geometry}.

    To generate our first examples, we recall the following theorem due to Bryant.
\begin{theorem}[{\cite[Theorem 5]{bryant-2000-calibrated-embeddings-in-the-special-lagrangian-and-coassociative-cases}}]\label{theorem-bryant-calabi-yau-special-lagrangian}
    Let $M$ be a closed real-analytic oriented  3-manifold. Then there exists a Calabi--Yau 6-manifold $N$ and an orientation-preserving isometric embedding $M \xhookrightarrow{} N$ such that the image is calibrated by the Special Lagrangian calibration $\omega_{\SL}$. 
\end{theorem}
    The Special Lagrangian calibration $\omega_{\SL}$ has comass equal to one everywhere by definition of Calabi--Yau manifolds and we refer to \cite{bryant-2000-calibrated-embeddings-in-the-special-lagrangian-and-coassociative-cases} for details on the construction and definitions. 
    
    Next, we compare this to Jormakka's classification of closed and oriented quasiregularly elliptic 3-manifolds \cite{jormakka-1988-the-existence-of-quasiregular-mappings-from-r3-to-closed-orientable-3-manifolds}.  
\begin{theorem}\label{theorem-jormakka}
    Up to an orientation-preserving diffeomorphism, the only closed, oriented quasiregularly elliptic 3-manifolds are $\S^3$, $\S^1 \times \S^2$, the torus $\mathbb{T}^3$, and their oriented quotients.
\end{theorem}

\begin{corollary}\label{corollary-reaping-the-harvest}
    Let $M$ be \(\S^3\), \(\S^1 \times \S^2\), \(\mathbb{T}^3\), or an oriented quotient thereof. Then there exists a Calabi--Yau 6-manifold $N$ and a quasiregular curve $F\colon \R^3 \to (N, \omega_{\SL})$, with bounded length distortion, whose image \(F(\mathbb{R}^3)\) is isometric to $M$, and which factors through a torus $\mathbb{T}^3$. In particular, the corresponding $(N,\omega_{\SL})$ are quasiregularly elliptic. 
\end{corollary}
\begin{proof}
    By \Cref{theorem-alexander-branched-cover}, there are quasiregular maps $f_n \colon \mathbb{T}^n \to \S^n$ that have bounded length distortion. This and \Cref{theorem-bryant-calabi-yau-special-lagrangian} imply the conclusion when $M$ is the sphere $\S^3$. Furthermore, clearly the map $f \colon \mathbb{T}^3 = \S^1 \times \mathbb{T}^2 \to \S^1 \times \S^2$, $(t,z) \mapsto (t, f_2(z))$ and the identity map $\mathbb{T}^3 \to \mathbb{T}^3$ provide examples of quasiregular mappings that have bounded length distortion. Clearly the oriented quotients of these spaces admit such maps as well by elementary covering space theory. The claim follows by applying this observation and \Cref{theorem-bryant-calabi-yau-special-lagrangian}. Note that the quasiregular ellipticity follows by considering a local isometric covering map $\R^3 \to \mathbb{T}^3$.
\end{proof}
    \begin{remark}
    Bryant's examples of Calabi--Yau manifolds are non-compact, and thus \Cref{corollary-reaping-the-harvest} gives examples of non-compact quasiregularly elliptic Calabi--Yau 6-manifolds \((N,\omega_{\SL})\). As the entire quasiregular curves produced in these examples factor through a torus, these calibrated manifolds do not satisfy Liouville's theorem in the sense of \Cref{theorem-liouville-property-equivalent-normality} (1).
    \end{remark}
    The following well-known proposition shows that the class of ``conformally elliptic'' manifolds is more restrictive than that of quasiregularly elliptic manifolds. In fact, examples already occur among the diffeomorphism equivalence classes of manifolds from Jormakka's classification.
\begin{proposition}\label{proposition-trivial-distortion-examples}
    Suppose that $M$ is an oriented $n$-manifold for which there exists a non-constant $1$-quasiregular map $F \colon \R^n \to M$ for $n \geq 3$. Then $F$ is smooth and its differential is non-degenerate. Furthermore, either $F$ is a smooth conformal covering map of $M$, or it extends to a smooth covering map $\mathbb{S}^{n} \simeq \mathbb{R}^n \cup \{\infty\} \to M$. In particular, $M$ is conformally flat.
\end{proposition}
\begin{remark}
    The proof shows that when $F$ in \Cref{proposition-trivial-distortion-examples} is not a covering map of $M$, it extends to a smooth conformal covering map $\S^n \simeq \R^n \cup \{\infty\} \to M$.
\end{remark}
    Above, an $n$-manifold is conformally flat if every point has an open neighbourhood which is conformally equivalent to a domain in $\R^n$. When $n  = 3$, the conformal flatness is equivalent to the vanishing of the \emph{Cotton tensor} while for $n \geq 4$, it is characterized by the vanishing of the \emph{Weyl tensor}.
\begin{proof}
    The argument we outline is quite standard, see e.g. \cite{bonk-heinonen-2001-quasiregular-mappings-and-cohomology}. We first recall that $F$ is a local homeomorphism by a localization to charts and a small distortion variant of Liouville's theorem \cite[Chapter VI, Theorem 8.14]{rickman-1993-quasiregular-mappings} and, thus, a local diffeomorphism by Ferrand's resolution of the Lichnerowicz conjecture \cite{lelong-ferrand-1976-geometrical-interpretations-of-scalar-curvature-and-regularity-of-conformal-homeomorphisms}. Then, by Zorich's theorem, the lift $\widehat{F} \colon \R^n \to \widehat{M}$ of $F$ to the universal cover of $M$ is a homeomorphism from $\mathbb{R}^n$ onto its image and the complement of the image in $\widehat{M}$ is zero-dimensional, see \cite[p.~336]{gromov-2007-metric-structures-for-riemannian-and-non-riemannian-spaces} or \cite{holopainen-pankka-2004-mappings-of-finite-distortion-global-homeomorphism-theorem} for this formulation. Since homeomorphisms preserve the number of ends, it follows that the complement of the image is at most a point. If the image omits a point, then $\widehat{F}$ extends to a conformal homeomorphism from $\mathbb{S}^n \simeq \mathbb{R}^n \cup \{\infty\}$ onto $\widehat{M}$ as a point singularity has zero $n$-capacity. By Ferrand's result above, the extension is a conformal diffeomorphism. This implies that $\widehat{M}$ is conformally diffeomorphic to the $n$-sphere $\mathbb{S}^n$. In case the image of $\widehat{F}$ does not omit any points, then $\widehat{F}$ is a conformal diffeomorphism onto $\widehat{M}$.
\end{proof}

    A natural follow-up to \Cref{corollary-reaping-the-harvest} is to ask whether the conclusion holds for \emph{closed} targets. In the following subsections, we compile a list of quasiregularly elliptic manifolds that can be (diffeomorphically) realized as a Special Lagrangian in a closed Calabi--Yau manifold, or in other geometries with special holonomy. However, we are not, nor can we be, exhaustive in our discussion, but we aim to illustrate the richness of this class of examples. We focus solely on positive examples of quasiregular ellipticity, noting that there are extensive lists of non-examples of varying topological types; see e.g. \cite{joyce-2000-compact-manifolds-with-special-holonomy} for further reading. For many of these submanifolds, their Betti numbers obstruct them from being quasiregularly elliptic \cite{prywes-2019-a-bound-on-the-cohomology-of-quasiregularly-elliptic-manifolds} (see also \cite{heikkila-pankka-2025-de-rham-algebras-of-closed-quasiregularly-elliptic-manifolds-are-euclidean,manin-prywes-2025-elliptic-quasiregularly-elliptic-manifolds}).

    \begin{table}[t]
    \begin{center}
    \begin{tabular}{|c|c|c|c|c|c|c|}
    \cline{1-1}
    \\
    Dimension three \\
    \\
    \hline
    \diagbox{calibration}{submanifold}  & $\S^3$ & $\mathbb{T}^3 $ & $\S^1 \times \S^2 $ & $\R P^3$ & $\R P^3 \# \R P^3$ & others \\
    \hline
    associative & $\checkmark$ &  & $\checkmark$ & $\checkmark$ & $\checkmark$ &  \\
    \hline
    Special Lagrangian & $\checkmark$ & $\checkmark$ &  &  &  &  \\
    \hline
    \multicolumn{7}{c}{} \\[0.5em]  
    \cline{1-1}
    \\
    Dimension four \\
    \\
    \hline
    \diagbox{calibration}{submanifold}  & $\S^4$ & $\mathbb{T}^4 $ & $\mathbb{T}^2 \times \S^2 $ & $\S^2 \times \S^2$ & $\C P^2 $ & others \\
    \hline
    coassociative &  & $\checkmark$ & $\checkmark$ &  &  &  \\
    \hline
    Special Lagrangian &  & $\checkmark$ &  &  &  &  \\
    \hline
    Kähler &  & $\checkmark$ &  &  & $\checkmark$ &  \\
    \hline
    Cayley &  & $\checkmark$ &  & $\checkmark$ &  &  \\
    \hline
    quaternionic & $\checkmark$ &  &  &  &  &  \\
    \hline
    \end{tabular}
    \end{center}
    \caption{A table summarizing the discussion of quasiregularly elliptic manifolds as submanifolds of closed calibrated manifolds.}\label{figure-quasiregular-ellipticity}
    \end{table}

    \subsection{Dimension two}
    The only closed quasiregularly elliptic manifolds are the sphere $\S^2$ and the torus $\mathbb{T}^2$. When the former is identified with the complex projective line $\C P^1$, we see that it admits an isometric embedding $\C P^1 \xhookrightarrow{} \C P^{\ell}$ for $\ell \geq 2$ using homogeneous coordinates. The image is calibrated by the Fubini--Study form \(\omega_{\FS}\), i.e. the canonical symplectic form on $\C P^\ell$. This was already observed in \cite{heikkila-2024-quasiregular-curves-and-cohomology}. The torus can either be embedded as a complex submanifold in $\mathbb{T}^{2\ell} \simeq \mathbb{C}^{\ell}/\mathbb{Z}^{2\ell}$, and thus is calibrated by the symplectic form \(\omega_{\mathrm{symp}}\), or as a Special Lagrangian submanifold in $\mathbb{T}^4 \simeq \mathbb{C}^{2}/\mathbb{Z}^{4}$. Therefore \((\CP^\ell, \omega_{\FS})\) and \((\mathbb{T}^{2\ell},\omega_{\mathrm{symp}})\) for \(\ell \geq 2\), and \((\mathbb{T}^4, \omega_{\SL})\) are quasiregularly elliptic.

    \subsection{Dimension three}
    We consider Jormakka's classification. The sphere $\mathbb{S}^3$ occurs as a calibrated submanifold in several settings of interest. Indeed, as discussed in the introduction, Joyce proposed, in \cite{joyce-2002-on-counting-special-lagrangian-homology-3-spheres}, an invariant of Calabi--Yau manifolds $(N, \omega_{\SL})$ based on \emph{counting} Special Lagrangian rational homology $3$-spheres. 
    When the fundamental group of a rational homology $3$-sphere \(M \subset N\) is finite, its universal cover \(\tilde{M}\) is diffeomorphic to the standard $3$-sphere by the Poincaré conjecture, and thus such a rational homology sphere can be interpreted as a quasiregular curve $\S^3 \to \tilde{M} \to M \xhookrightarrow{} ( N, \omega_{\SL} )$. When such an $M$ exists, clearly $( N, \omega_{\SL} )$ is quasiregularly elliptic. See \cite{joyce-2018-conjectures-on-counting-associative-4-folds-in-g2-manifolds} for discussion about the similar role which associative rational homology $3$-spheres play for $G_2$ manifolds. 
    
    Observe that the torus $\mathbb{T}^3$ can also be realized as a Special Lagrangian, e.g. in $\mathbb{T}^6$, similarly to the two-dimensional case. Moreover, the famous SYZ conjecture on Calabi--Yau manifolds (see \cite[Section 9]{joyce-2007-riemannian-holonomy-groups-and-calibrated-geometry}) is related to the existence of (possibly singular) fibrations with generic fibers being Special Lagrangian and diffeomorphic to $\mathbb{T}^3$ on (closed) Calabi--Yau manifolds. This would obviously create a large supply of quasiregularly elliptic Calabi--Yau manifolds. For example, the Kreuzer--Skarke list \cite[page 2]{altman-gray-he-jejjala-nelson-2015-a-calabi-yau-database-threefolds-constructed-from-the-kreuzer-skarke-list} contains at least 30,108 distinct Calabi--Yau 6-manifolds. There are examples of \(\S^1 \times \S^2\) in \(G_2\) manifolds calibrated by the associative form; see e.g. \cite[Lemma 12.6.3]{joyce-2000-compact-manifolds-with-special-holonomy} or \cite[Lemma 12.7.3]{joyce-2000-compact-manifolds-with-special-holonomy}. For further topological types which the ambient $G_2$ manifold can have, see the last column of \cite[Table 12.5]{joyce-2000-compact-manifolds-with-special-holonomy}.
    
    To finish this list of examples, see \cite[Example 7.7]{bera-2026-associative-submanifolds-in-twisted-connected-sum-g2-manifolds} for an associative real projective space $\R P^3$ (double covered by $\S^3$) and see \cite[Examples 7.8--7.9]{bera-2026-associative-submanifolds-in-twisted-connected-sum-g2-manifolds} for an associative connected sum $\R P^3 \# \R P^3$ (double covered by $\S^1 \times \S^2$ \cite[Section 4]{scott-1983-the-geometries-of-3-manifolds}). \Cref{figure-quasiregular-ellipticity} summarizes this list.

    \subsection{Dimension four}
    The topological classification of closed quasiregularly elliptic 4-manifolds was recently completed by Heikkilä--Pankka \cite{heikkila-pankka-2025-de-rham-algebras-of-closed-quasiregularly-elliptic-manifolds-are-euclidean}, Piergallini--Zuddas \cite{piergallini-zuddas-2021-branched-coverings-of-basic-4-manifolds}, and Manin--Prywes \cite{manin-prywes-2025-elliptic-quasiregularly-elliptic-manifolds}. The complete list consists of the following spaces: the complex projective plane $\mathbb{C}P^2$, the torus $\mathbb{T}^4$, the sphere $\mathbb{S}^4$, the products $\mathbb{T}^2 \times \mathbb{S}^2$, $\mathbb{S}^2 \times \mathbb{S}^2$, $\mathbb{S}^1 \times \mathbb{S}^3$, the connected sums $\#^{k} \CP^2 \#^\ell \overline{\CP}^2$ for $0 \leq k, \ell \leq 3$ and $\#^{k}( \S^2 \times \S^2 )$ for $1 \leq k \leq 3$, and the oriented quotients of these spaces. Here $\overline{\CP}^2$ is the standard complex projective space equipped with the opposite orientation. 
    
    We have located the following examples in the literature: 
    Analogously to $\S^2 \simeq \C P^1$, the sphere $\S^4$ is diffeomorphic to the quaternionic projective line $\mathbb{H}P^1$ which admits a canonical embedding into $\mathbb{H}P^{\ell}$ for $\ell \geq 2$. Its image is calibrated by the quaternionic calibration (the Kraines form \cite{kraines-1966-topology-of-quaternionic-manifolds}). As in the case of $\mathbb{C}P^1$, $\mathbb{C}P^2$ occurs as a complex submanifold of $\mathbb{C}P^{\ell}$ for $\ell \geq 3$ which is therefore calibrated by a suitably normalized power of the Fubini--Study form. As in dimension two, the torus $\mathbb{T}^4$ occurs as a Kähler or Special Lagrangian submanifold in $\mathbb{T}^{8}$. It can also be realized in a three-parameter family of coassociative examples of $G_2$ manifolds \cite[Example 12.6.4]{joyce-2000-compact-manifolds-with-special-holonomy}, and as a Cayley submanifold (calibrated by the Cayley form) \cite[Example 14.3.1]{joyce-2000-compact-manifolds-with-special-holonomy}. 
    Similarly to Special Lagrangian torus bundles on Calabi--Yau manifolds, the coassociative torus bundles play a role in $G_2$ manifolds; see \cite{gukov-yau-zaslow-2003-duality-and-fibrations-on-g2-manifolds}. The product $\mathbb{T}^2 \times \S^2$ occurs as a coassociative submanifold in a one-parameter family of examples, cf. \cite[Example 12.6.6]{joyce-2000-compact-manifolds-with-special-holonomy}, and the product $\S^2 \times \S^2$ occurs as a Cayley submanifold \cite[Example 14.3.4]{joyce-2000-compact-manifolds-with-special-holonomy} (for discussion on the deformation theory of Cayley submanifolds, see \cite[Section 12.5.1]{joyce-2007-riemannian-holonomy-groups-and-calibrated-geometry}). 
    
    The authors are not aware of any examples of the product $\mathbb{S}^1 \times \mathbb{S}^3$ or the connected sums $\#^{k} \CP^2 \#^\ell \overline{\CP}^2$ for $0 \leq k, \ell \leq 3$ and $\#^{k}( \S^2 \times \S^2 )$ for $1 \leq k \leq 3$ appearing as e.g. coassociative or Cayley submanifolds. However, for any closed quasiregularly elliptic manifold $M$ and any closed manifold $\Sigma$, clearly $N = M \times \Sigma$ has $M \times \{p\}$ as a calibrated submanifold for $\pi_{M}^{*}\vol_{M}$ and any $p \in \Sigma$, where \(\pi_M\colon N \to M\) is the coordinate projection to \(M\). It is unclear if all (or any) of these can be realized as calibrated submanifolds of a compact space that is not of product type.
    The discussion above is summarized in \Cref{figure-quasiregular-ellipticity}.

\newcommand{\etalchar}[1]{$^{#1}$}
\providecommand{\bysame}{\leavevmode\hbox to3em{\hrulefill}\thinspace}
\providecommand{\MR}{\relax\ifhmode\unskip\space\fi MR }
\providecommand{\MRhref}[2]{%
  \href{http://www.ams.org/mathscinet-getitem?mr=#1}{#2}
}
\providecommand{\href}[2]{#2}

\end{document}